\documentclass[12pt]{amsart}
\usepackage{amscd,graphicx,epsfig}
\usepackage[colorlinks=true, pdfstartview=FitV, linkcolor=blue, citecolor=blue, urlcolor=blue]{hyperref}
\usepackage{amssymb}
\usepackage{xcolor}
\usepackage{tikz-cd}

\usepackage[english]{babel}
\usepackage[utf8]{inputenc}

\title[Algebraically generated groups and their Lie algebras]{Algebraically generated groups \\ and their Lie algebras}
\author{Hanspeter Kraft and Mikhail Zaidenberg}

\address{Departement Mathematik und Informatik,
\newline\indent Universit\"at Basel,  Spiegelgasse 1, CH-4051 Basel, Switzerland}
\email{hanspeter.kraft@unibas.ch}
\address{Univ. Grenoble Alpes, CNRS, IF, 38000 Grenoble, France}
\email{mikhail.zaidenberg@univ-grenoble-alpes.fr}


\thanks{}

\newtheorem{thm}{Theorem}[subsection]
\newtheorem*{thm*}{Theorem}
\newtheorem*{thmA}{Theorem A}
\newtheorem*{thmB}{Theorem B}
\newtheorem*{thmC}{Theorem C}
\newtheorem*{thmD}{Theorem D}
\newtheorem*{thmE}{Theorem E}

\newtheorem{prop}[thm]{Proposition}
\newtheorem{lem}[thm]{Lemma}
\newtheorem{cor}[thm]{Corollary}
\newtheorem*{cor*}{Corollary}

\newtheorem{que}{Question}

\theoremstyle{definition}
\newtheorem{defn}[thm]{Definition}
\newtheorem{exa}[thm]{Example}

\theoremstyle{remark}
\newtheorem*{rem*}{Remark}
\newtheorem{rem}[thm]{Remark}
\newtheorem*{rems*}{Remarks}
\newtheorem{rems}[thm]{Remarks}

\newcommand{\name}[1]{{\rm\textsc{#1\/}}}
\newcommand{\NN}{{\mathbb N}}
\newcommand{\ZZ}{{\mathbb Z}}
\newcommand{\CC}{{\mathbb C}}
\newcommand{\QQ}{{\mathbb Q}}
\newcommand{\kk}{{\Bbbk}}

\newcommand{\GG}{{\mathbb G}}
\newcommand{\kst}{{\kk^*}}
\renewcommand{\AA}{{\mathbb A}}
\newcommand{\An}{{\mathbb A}^{n}}
\newcommand{\Atwo}{{\mathbb A}^{2}}
\newcommand{\Aone}{{\mathbb A}^{1}}
\newcommand{\SSS}{\mathcal S}

\newcommand{\JJJ}{\mathcal J}
\newcommand{\FFF}{\mathcal F}

\newcommand{\KKK}{\mathcal{K}}
\newcommand{\TTT}{\mathcal{T}}
\newcommand{\simto}{\xrightarrow{\sim}}
\newcommand{\be}{\begin{enumerate}}
\newcommand{\ee}{\end{enumerate}}
\newcommand{\beq}{\begin{equation}}
\newcommand{\eeq}{\end{equation}}
\newcommand{\bullitem}{\item[$\bullet$]}

\newcommand{\G}{\mathfrak G}
\newcommand{\F}{\mathfrak F}
\newcommand{\V}{\mathfrak V}

\newcommand{\X}{\mathfrak X}
\renewcommand{\H}{\mathfrak H}
\newcommand{\A}{\mathfrak A}
\newcommand{\J}{\mathfrak J}
\newcommand{\B}{\mathfrak B}

\newcommand{\into}{\hookrightarrow}
\newcommand{\p}{\partial}

\newcommand{\ps}{\par\smallskip}
\newcommand{\Ga}{\GG_a}
\newcommand{\Gm}{\GG_m}
\newcommand{\tp}[2]{\texorpdfstring{#1}{#2}}

\DeclareMathOperator{\id}{id}

\DeclareMathOperator{\Aut}{Aut}
\DeclareMathOperator{\SAut}{SAut}
\DeclareMathOperator{\Lie}{Lie}
\DeclareMathOperator{\VF}{Vec}
\DeclareMathOperator{\Div}{div}
\DeclareMathOperator{\ad}{ad}
\DeclareMathOperator{\GL}{GL}
\DeclareMathOperator{\SL}{SL}
\DeclareMathOperator{\SLtwo}{SL_{2}}
\DeclareMathOperator{\Norm}{Norm}
\DeclareMathOperator{\Ad}{Ad}

\DeclareMathOperator{\Int}{Int}
\DeclareMathOperator{\Der}{Der}
\DeclareMathOperator{\Aff}{Aff}

\DeclareMathOperator{\Jac}{jac}
\DeclareMathOperator{\pr}{pr}
\DeclareMathOperator{\Spec}{Spec}
\DeclareMathOperator{\ev}{ev}
\DeclareMathOperator{\Jonq}{Jonq}
\DeclareMathOperator{\Cent}{Cent}
\DeclareMathOperator{\cent}{\mathfrak{cent}}
\DeclareMathOperator{\Ll}{\mathcal L}
\DeclareMathOperator{\Trans}{Trans}
\DeclareMathOperator{\Span}{span}
\DeclareMathOperator{\rad}{rad}
\DeclareMathOperator{\gl}{\mathfrak{gl}}

\DeclareMathOperator{\Grass}{Grass}

\newcommand{\reg}{\text{\it reg}}
\renewcommand{\subset}{\subseteq}
\renewcommand{\supset}{\supseteq}
\newcommand{\bg}{\mathbf{g}}
\newcommand{\bu}{\mathbf{u}}
\newcommand{\bv}{\mathbf{v}}

\newcommand{\OOO}{\mathcal O}
\newcommand{\ii}{\boldsymbol i}
\newcommand{\quot}{/\!\!/}
\renewcommand{\phi}{\varphi}

\newcommand{\bb}{\mathfrak b}
\newcommand{\nn}{\mathfrak n}

\begin{document}
{\small
\begin{abstract}
The automorphism group $\Aut(X)$ of an affine variety $X$ is an ind-group. Its 
Lie algebra is canonically embedded into the Lie algebra $\VF(X)$ of vector fields on $X$.
We study the relations between subgroups of $\Aut(X)$ and Lie subalgebras of $\VF(X)$.

We show that a subgroup $G\subset \Aut(X)$ generated by  a family of
connected algebraic subgroups $G_i$ of $\Aut(X)$
is algebraic if and only if the Lie algebras $\Lie G_i \subset \VF(X)$ generate a finite dimensional 
Lie subalgebra of $\VF(X)$.

Extending  a result by \name{Cohen-Draisma} 
\cite{CoDr2003From-Lie-algebras-} 
we prove that a locally finite Lie algebra $L \subset \VF(X)$ generated by locally nilpotent vector fields is algebraic,  i.e. $L = \Lie G$ for an algebraic subgroup  $G \subset \Aut(X)$.

Along the same lines we prove that if a subgroup $G \subset \Aut(X)$ generated by finitely many connected  algebraic groups is solvable, then it is an algebraic group. 

We also show that a unipotent algebraic subgroup $U \subset \Aut(X)$ has derived length $\leq \dim X$. This result is based on the following triangulation theorem: 
{\it Every unipotent algebraic subgroup of $\Aut(\An)$ with a dense orbit in $\An$ is conjugate to a subgroup of the de Jonqui\`eres subgroup.}

Furthermore, we give an example of a free subgroup $F\subset \Aut(\Atwo)$ 
generated by two algebraic
elements such that the Zariski closure $\overline{F}$ is a free product of two nested commutative closed unipotent  ind-subgroups. 

To any affine ind-group $\G$ one can associate a canonical ideal  $L_\G \subset \Lie\G$. It is linearly generated by the tangent spaces $T_e X$ for all algebraic subsets $X \subset \G$ which are smooth in $e$. It has the important property that for a surjective homomorphism $\phi\colon \G \to \H$ the induced homomorphism $d\phi_e\colon L_\G \to L_\H$ is surjective as well. Moreover, if $\H \subset \G$ is a subnormal closed ind-subgroup of finite codimension, then $L_\H$ has finite codimension in $L_\G$.
\end{abstract}
}

\maketitle
{\small
\noindent
%
{\small
\tableofcontents}

\section{Introduction and main results}
The introduction contains the necessary preliminaries and definitions, gives some background material and describes the main results of the paper. The details then follow in the subsequent sections.
\subsection{Notation}\label{notation:sec}
Our base field  $\kk$ is algebraically closed and of characteristic zero. $\An$ stands for affine $n$-space over $\kk$,  and $\Ga := \kk^+$ and $\Gm:=\kk^*:=\kk\setminus\{0\}$ denote the {\it additive} and the {\it multiplicative} groups of $\kk$. An algebraic group is always an {\it affine algebraic group}, and 
a {\it $G$-variety $X$} is an affine variety with an action of $G$
such that the corresponding map $G \times X \to X$ is a morphism. 

For every 
$G$-variety $X$ there is a canonical \mbox{(anti-)} homomorphism of Lie algebras $\xi\colon \Lie G \to \VF(X)$,
where $\VF(X)$ stands for the Lie algebra of (algebraic) vector fields on $X$. The construction of this 
anti-homomorphism goes as follows. For any $x \in X$ let $\theta_x \colon G \to X$ be the orbit map $g \mapsto gx$, and denote by 
\[\tag{$*$}
d\theta_x\colon \Lie G:=T_eG  \to T_xX
\]
its differential in $e \in G$.  If $A \in \Lie G$, then the corresponding vector field $\xi_A$ is given by 
$\xi_A(x):= d\theta_x(A)\in T_xX$. If $L \subset \VF(X)$ denotes the image of $\Lie G$, then one has  $T_xGx = L(x)$ for any $x \in X$. 
Moreover, the differential $(d\theta_x)_g\colon T_gG \to T_{gx}Gx$ of the orbit map $\theta_x$ is surjective in every $g \in G$.
(For a reference one might look at \cite[Appendix A.4]{Kr2016Algebraic-Transfor}.)

For a $G$-variety $X$ we have linear actions of  $G$ on the coordinate ring $\OOO(X)$ and on the vector fields $\VF(X)$,
and these representations are {\it locally finite and rational}, i.e. every element is contained in
a finite dimensional $G$-invariant subspace and the linear action of $G$ on any finite dimensional $G$-invariant subspace is rational. Moreover, the homomorphism $\xi\colon \Lie G \to \VF(X)$ is $G$-equivariant, see \cite[Sect.~7.3]{FuKr2018On-the-geometry-of}.

\ps
\subsection{Automorphism groups and vector fields}\label{VF.subsec}
For an affine variety $X$ the group of regular automorphisms $\Aut(X)$ is an {\it affine ind-group}. We refer to the paper \cite{FuKr2018On-the-geometry-of} 
for an introduction to ind-varieties and ind-groups and for basic concepts,  cf. \cite{Sh1981On-some-infinite-d, Ku2002Kac-Moody-groups-t}. 
For an affine ind-group $\G$ the tangent space $T_e\G$ carries a natural structure of a Lie algebra which will be denoted by $\Lie \G$. In case of 
$\Aut(X)$ there is a canonical embedding $\xi\colon\Lie\Aut(X) \into \VF(X)$ which is an anti-homomorphism of Lie algebras. It is constructed in a similar way as explained above for an algebraic group $G$,  see \cite[Proposition~7.2.4]{FuKr2018On-the-geometry-of}. It has the following property:

{\it If $A \in \Lie \Aut(X)$ and $\xi_A\in\VF(X)$ the corresponding vector field, then one has, for every $x \in X$,
$$
\xi_A(x) = d(\theta_x)_e(A) \ \text{ where\ } \ \theta_x\colon \Aut(X) \to X \text{ is the orbit map in }x.
$$}
\par\noindent
The group $\G$ acts on $\Lie \G$ by the adjoint action: $\Ad(g) := d(\Int g)_e$. For $\G = \Aut(X)$ this action is induced by the action on all vector fields $\VF(X)$ which is given by 
$$
\Ad g(\delta)(gx) := d\mu_g\,\delta(x), \ \text{ i.e. }\  \Ad g(\delta) = d\mu_g \circ \delta \circ \mu_{g^{-1}}
$$
where $\mu_g\colon X \simto X$ denotes the multiplication map $x \mapsto gx$. Here we consider a vector field $\delta \in \VF(X)$ as a section of the ``tangent bundle'' $\TTT_X \to X$, and then $\Ad g(\delta)$ is defined by the following diagram:
$$
\begin{CD}
\TTT_X @<{d\mu_g}<< \TTT_X \\
@A{\Ad g(\delta)}AA @AA{\delta}A \\
X @>{\mu_{g}^{-1}}>>X
\end{CD}
$$
It is easy to see that the embedding $\xi\colon \Lie\Aut(X) \into \VF(X)$ defined above is $\Aut(X)$-equivariant:
$$
\xi_{\Ad(g) A} = \Ad g(\xi_A) \ \text{ or } \ \xi_{\Ad(g)A}(gx) = d\mu_g \, \xi_A(x).
$$
Ind-varieties and ind-groups carry a natural topology, 
called {\it Zariski-topology} or {\it ind-topology}. All topological notation in this paper will be with respect to this topology. 
\par\smallskip
A subset $Y \subset \V$ of an ind-variety $\V=\bigcup_k\V_k$ is called {\it bounded} if it is contained in $\V_k$ for some $k$. It is called {\it algebraic} if it is bounded and locally closed.
If $X$ is a $G$-variety, then the canonical map $G \to \Aut(X)$ is a homomorphism of ind-groups, and the image of $G$ is closed and algebraic, 
see  \cite[Proposition~2.7.1]{FuKr2018On-the-geometry-of}. In particular, every algebraic subgroup $G \subset \Aut(X)$ is closed and thus a linear algebraic group. It follows that $\Lie G \subset \Lie\Aut(X)$ is a Lie subalgebra and that $\Lie G$ is canonically embedded into the vector fields $\VF(X)$. 

\ps
\noindent
{\bf Convention.} In the following our ind-groups and ind-varieties will always be {\it affine\/} ind-groups and {\it affine\/} ind-varieties. By abuse of notation, we will constantly identify, for an algebraic group $G \subset \Aut(X)$, the Lie algebra $\Lie G$ with its image $\xi(\Lie G) \subset \VF(X)$, although the map $\xi$ is an anti-homomorphism.
\ps
\subsection{Algebraically generated groups}\label{alg-gen-groups:sec} 
A first result showing a very strong relation between the Lie algebra of an ind-group and the group itself is the following, see \cite[Proposition~2.2.1(3) and (4)]{FuKr2018On-the-geometry-of}.
\begin{prop}
Let $\G$ be an ind-group. Then the connected component $\G^\circ$ is an algebraic group if and only if $\Lie \G$ is finite dimensional.
\end{prop}
We will prove a similar statement for so-called algebraically generated groups, see \cite{ArFlKa2013Flexible-varieties}.
Let $(G_{i})_{i\in I}$ be a family of connected algebraic subgroups of an ind-group $\G$. 
The subgroup $G\subset \G$  generated by the $G_{i}$ will be called  {\it algebraically generated} (Definition~\ref{AGgroup.def}).

Let $L(G) \subset \Lie\G$ be the Lie subalgebra generated by the Lie algebras $\Lie G_{i}$.
We will see in Theorem~\ref{main1.thm} that $L(G)$
is invariant under the action of $G$ and of the closure $\overline{G}$. This implies that
$L(G) \subset \Lie\overline{G}$ is an ideal.

\begin{que}\label{Q1}
Do we have $L(G) = \Lie \overline{G}$?
\end{que}
Our result in this direction is the following, see Theorem~\ref{main2.thm}.

\begin{thmA}\label{mthm:0}
The subgroup $G\subset \G$ is an algebraic group if and only if $L(G)$ is finite dimensional, and in this case we have $\Lie G = L(G)$.
\end{thmA}

See also Theorem~\ref{main1.thm}.
Our example in section~\ref{sec:example} gives an infinite dimensional $L$ where Question~\ref{Q1} above has a positive answer, see Remark~\ref{example.rems}(1).

Theorem~A  above is an important ingredient in the proof of the following result.
\begin{thmB}
Let $G \subset \G$ be a solvable subgroup generated by a family of connected algebraic subgroups. If the family is finite, then $G$ is an algebraic group. In general, $G$ is nested, i.e. a filtered union of algebraic subgroups, and is of the form $U_G \rtimes T$ where $T$ is a torus and $U_G$ a nested unipotent group.

In particular, if $G$ is generated by a family of unipotent algebraic groups, then $G = U_G$ is a nested unipotent group.
\end{thmB}
The proof of Theorem B will be given in Section~\ref{solvable.subsec}, see Theorem~\ref{thmB-new}.
\begin{rems}\label{AG.rems}
\be
\item 
Clearly, a connected nested ind-group is algebraically generated.

\item
If $G$ is an algebraic group and $H_1,\ldots, H_m \subset G$ a finite set of connected closed subgroups, then the subgroup $H$ generated by the $H_i$ is closed. (Indeed, there is a sequence $i_1,i_2,\ldots,i_N$ such that the product  $P:=H_{i_1} H_{i_2} \cdots H_{i_N}$ is dense (and constructible) in $\overline{H}$. It then follows that $P P = \overline{H}$, hence $H = \overline{H}$.)

As a consequence we see that a subgroup of a nested ind-group $\G$ generated by finitely many connected algebraic groups is an algebraic group. 

Note that this statement does not hold for non-connected subgroups. E.g. the two matrices 
$A:=\left[\begin{smallmatrix} 0 & -1\\1 & \phantom{-}0\end{smallmatrix}\right]$ and 
$B:=\left[\begin{smallmatrix} 0 & -\frac{1}{2}\\2 & \phantom{-}0\end{smallmatrix}\right]$ have both order 4, and the product $AB = \left[\begin{smallmatrix} -2 & \phantom{-}0 \\ \phantom{-}0 &-\frac{1}{2}\end{smallmatrix}\right]$ has infinite order. Thus $\overline{\langle A, B \rangle} = N$, the normalizer of the diagonal torus, but $\langle A, B \rangle \subseteq \SL_2(\QQ)$. It follows that the subgroup of $\SL_2(\kk)$ generated by the finite groups $\langle A\rangle$ and $\langle B\rangle$ is not closed.

\item
If the ind-group $\G$ verifies \name{Tits}' alternative, then the assumption of solvability of $G$ in Theorem B can be replaced by the weaker assumption that $G$ has no nonabelian free subgroup; see e.g. \cite{ArZa2021Tits-type-alternative-for} and the literature therein for a discussion of \name{Tits}' alternative in ind-groups of type $\Aut(X)$. 
\ee
\end{rems}
\ps
\subsection{Solvability and triangulation}\label{ss:solvable-triangular}\label{solvability-triangulation.subsec}
The de Jonqui\`eres subgroup $\Jonq(n) \subset\Aut(\An)$ consists of the automorphisms of the form
$$
\phi=(a_1x_1+f_1, a_2x_2 +f_2(x_1),
\ldots, 
a_n x_n + f_n(x_1,\ldots,x_{n-1}))
$$
where $a_i \in \kst$ and $f_i\in\kk[x_1,\ldots,x_{i-1}]$. 
It is known that the unipotent elements of $\Jonq(n)$ form a solvable subgroup of derived length $n$ which is not nilpotent for $n>1$ (see Remarks~\ref{rem1} and \ref{rem2}). More generally, we have the following result.

\begin{thmC}
A nested unipotent subgroup $U \subset \Aut(X)$ is 
solvable of derived length $\leq \max\{\dim Ux\mid x \in X\}\leq \dim X$.
\end{thmC}
It is known that the derived length of a nilpotent (respectively, a solvable) connected Lie group $G$ acting faithfully on a Hausdorff topological space $M$ is bounded above by $\dim M$ (respectively, by $\dim M+1$), 
see \cite[Theorem 1.2]{EpTh1979Transformation-groups}. Up to passing to a finite index subgroup, the former estimate works as well in the case of a finitely generated nilpotent group acting faithfully on a quasi-projective variety $X$ defined over a field of characteristic zero \cite[Theorem B]{Ab2022Actions-of-nilpote}. 
\ps
The proof of Theorem~C will be given in Section~\ref{nested-unipotent.subsec}, see Theorem~\ref{thmC}. It is based on the following important triangulation result, see Theorem~\ref{thmD}.
\begin{thmD}
Assume that a nested unipotent subgroup $U \subset \Aut(\An)$ has a dense orbit in $\An$. Then $U$ acts transitively on $\An$ and $U$ is triangulable, i.e., it is conjugate to a subgroup of the de Jonqui\`eres subgroup.
\end{thmD}

In \cite{Ba1984A-nontriangular-ac} \name{Bass} gave an example of a $\Ga$-action on $\AA^3$ which is not triangulable, i.e. the image $A \subset\Aut(\AA^3)$ of $\Ga$ is not conjugate to a subgroup of $\Jonq(3)$. The reason is that the fixed points set $(\AA^3)^{\Ga}$ is a hypersurface with an isolated singularity. This is not possible for a triangulable action, because for such an action the fixed point set has the form $X \times\Aone$. With the same idea one can construct  non-triangulable $\Ga$-actions in any dimension, see \cite{Po1987On-actions-of-bf-G}.

The $\Ga$-action of \name{Bass} corresponds to the locally nilpotent vector field $\delta:= (xz+y^2)(-2y\partial_x + z \partial_y)$, and the image $A$ of $\Ga$ contains the famous Nagata automorphism 
$$
\eta = (x-2y(xz+y^2) - z(xz+y^2)^2, y+z(xz+y^2), z). 
$$
Due to the celebrated \name{Shestakov-Umirbaev} Theorem  (\cite{ShUm2003The-Nagata-automor}) this $\Ga$-sub\-group is not contained in the tame subgroup  of $\Aut(\AA^3)$. However, it becomes tame in $\AA^4$, see \cite{Sm1989Stably-tame-automo}, but it is still non-triangulable in $\AA^4$, see \cite[Lemma~3.36]{Fr2006Algebraic-theory-o}.
The question arises if this action is {\it stably triangulable}, i.e. becomes triangulable in $\AA^3\times \AA^d$ for a suitable $d\geq 1$. In this context we have the following negative answer.

\begin{prop}\label{Bass.prop}
Consider a $\Ga$-action on $\An$ and assume that the fixed point set is a hypersurface with an isolated singularity. Then the action is not stably triangulable.
\end{prop}
\begin{proof}
(a)
Denote by $F \subset \An$ the fixed point set and by $p \in F$ the isolated singularity. For the quotient morphism $\pi\colon\An \to \An\quot\Ga$ the fiber $\pi^{-1}(\pi(p))$ has dimension $\geq 1$ and thus is not contained in the singularities of $F$. If we extend the action to $\An\times \AA^d$ where $\Ga$ acts trivially on $\AA^d$, then $(\An\times\AA^d)\quot \Ga = \An\quot\Ga \times \AA^d$ and the quotient morphism is equal to $\tilde\pi := \pi\times\id$. Moreover, the fixed point set is $\tilde F := F\times \AA^d$, and the singularities of $\tilde F$ are ${\tilde F}_{\text{sing}} := \{p\}\times \AA^d$. It follows that ${\tilde\pi}^{-1}(\tilde\pi({\tilde F}_{\text{sing}}))$ is not contained in ${\tilde F}_{\text{sing}}$.
\ps
(b)
Now assume that the action on $\AA^n \times \AA^d$ is triangulable, so that the corresponding vector field is equivalent to one of the form
$$
\tilde\delta = f_1\partial_{x_1} + f_2\partial_{x_2}  + \cdots + f_{m}\partial_{x_{m}}
$$ 
where $m := n + d$ and $f_i \in \kk[x_1,\ldots,x_{i-1}]$. The fixed point set $\tilde F$ is an irreducible hypersurface defined by an invariant (irreducible) function $h$, and so $h$ must divide all  the $f_i$. Since the singular points of $\tilde F$ form a subvariety of dimension $d$ this implies that  $f_1=\cdots=f_n = 0$. In fact, if $h$ would depend on less than $n$ variables, then the zero set  $\{h=0\}\subset\AA^m$ would have the form $X \times \AA^{d+1}$ and so the singular set of $\tilde F$ would have at least dimension $d+1$.

Let $f_{r+1}$ be the first non-zero coefficient of $\tilde\delta$. Then $r\geq n$ and $x_1,x_2,\ldots,x_r$ are invariants. Since $h$ divides $f_{r+1}$ we see that $h$ only depends on the variables $x_1,\ldots,x_r$. It follows that the linear $\Ga$-invariant morphism $\phi\colon \AA^m \to \AA^r$, $(x_1,\ldots,x_m) \mapsto (x_1,\ldots,x_r)$, maps the fixed point set $\tilde F$ to the hypersurface $F_r:=\{h=0\} \subset \AA^r$, and  
$\tilde F = F_r \times \AA^{m-r}$. This shows that $\phi^{-1}(\phi({\tilde F}_{\text{sing}})) = {\tilde F}_{\text{sing}}$. Since the invariant map $\phi$ factors through the quotient map $\tilde\pi$ this contradicts what we have seen in (a).
\end{proof}
\begin{rems}\label{freundenburg.rem}$\,$

\be
\item 
The non-triangulable $\Ga$-actions on $\An$ constructed  in \cite{Po1987On-actions-of-bf-G} are not stably triangulable.
\item
The Nagata automorphism $\eta$ is not contained in a unipotent subgroup $U \subset \Aut(\AA^3)$ which has a dense orbit in $\AA^3$, by our Theorem~D.
\item
If we consider the diagonal action on $\AA^3 \times \Aone$ where $\Ga$ acts on $\AA^3$ as in \name{Bass}' example and on $\Aone$ by translation, then this new action is triangulable. Indeed, we have the following result. 
\ps
{\it Let $U$ be a unipotent group acting on $\An$. Then the diagonal action on $\An \times U$ where $U$ acts by left multiplication on $U$ is triangulable.}
\ps
Indeed, consider the isomorphism $\phi\colon \An \times U \simto \An \times U$ given by $(x,u ) \mapsto (u^{-1} x,u)$. This morphism is $U$-equivariant with respect to the diagonal action on the left hand side and the action on $U$ by left multiplication (and the trivial action on $\An$) on the right hand side. Since the action on $U$ is triangulable by our Theorem~D above the claim follows.

Likewise, given any $\kk$-algebra $R$ and a locally nilpotent deri\-vation $\partial$ of $R$, its extension $\overline\partial$  to $R[x]$ by letting $\overline\partial(x)=1$ is conjugate to the locally nilpotent $R$-derivation $\delta$  of $R[x]$ defined by $\delta(x)=1$.
\ee
\end{rems}
\ps
\subsection{Locally finite subsets} 
Let $V$ be a vector space over $\kk$. 
We denote by $\Ll(V)$ the algebra of linear endomorphisms of $V$. 
An endomorphism $\lambda \in \Ll(V)$ is called {\it locally finite\/} if  the linear span $\langle \lambda^j(v) \mid j \in \NN\rangle$ 
is finite dimensional for any $v \in V$.
It is called {\it semisimple\/} if there is a basis of eigenvectors, and {\it locally nilpotent} if for any $v \in V$ there is an $m \in \NN$ such that $\lambda^m(v) = 0$. 
Every locally finite endomorphism $\lambda$ has a uniquely defined {\it additive Jordan-decomposition\/} $\lambda = \lambda_s + \lambda_n$ 
where $\lambda_s$ is semisimple, $\lambda_n$ locally nilpotent, and $\lambda_s \circ \lambda_n = \lambda_n \circ \lambda_s$.

A subset $S \subset \Ll(V)$ is called {\it locally finite} if every $v\in V$
is contained in a finite dimensional subspace $W \subset V$ which is invariant
under all elements from $S$. If $S \subset\GL(V)$, then $S$ is locally finite if
and only if the group $\langle S \rangle$ generated by $S$ is locally finite.

Note that a locally finite subspace $A \subset \Ll(V)$ is not necessarily finite dimensional. 
In fact, let $(e_1,e_2,\ldots)$ be a basis of $V:=\kk^\infty$ and define $\lambda_k\in\Ll(V)$ by $\lambda_k(e_j):=\delta_{kj}e_j$. 
Then $A:=\bigoplus_k \kk\lambda_k$ is an infinite dimensional locally finite subspace. 

\begin{defn}
Let $X$ be an affine variety. A morphism  $\phi\colon X \to X$ is called
{\it locally finite} if the induced endomorphism $\phi^{*}\in\Ll(\OOO(X))$ 
is locally finite. It is called {\it semisimple\/} if $\phi^*$ is semisimple, and {\it locally nilpotent} if $\phi^*$ is locally nilpotent.

A subgroup $H \subset \Aut(X)$ is called {\it locally finite} if the image of $H$
in $\Ll(\OOO(X))$ is locally finite.
\end{defn}
Note that a subgroup $H \subset \Aut(X)$ is locally finite if and only if the
closure $\overline{H} \subset \Aut(X)$ is an algebraic group. If $g \in \Aut(X)$ is locally finite, 
then the closed commutative subgroup $\overline{\langle g \rangle}$ has the form $\Gm^p \times \Ga^q \times F$ where $p\in\ZZ_{\geq 0}$, $q \in \{0,1\}$ and $F$ is a finite cyclic group.

There are many examples of subgroups $G \subset \Aut(X)$ that are not locally finite while being generated by locally finite elements. 
We will discuss such an example in Section~\ref{sec:example}. We also refer to the interesting discussions in  
\cite[9.4.3--9.4.5]{FuKr2018On-the-geometry-of} and \cite{PeRe2022When-is-the-automo} of subgroups $G \subset\Aut(X)$ which consist of locally finite elements.
\ps
\subsection{Locally finite vector fields}\label{lofi-VF.subsec}
Recall that the vector fields $\VF(X)$ are the derivations $\Der(\OOO(X))$ of $\OOO(X)$. Since  $\Der(\OOO(X)) \subset \Ll(\OOO(X))$ we can talk about {\it locally finite vector fields\/} and {\it locally finite subspaces\/} of $\VF(X)$.

\begin{exa}
If $G \subset \Aut(X)$ is an algebraic subgroup, then its Lie algebra $\Lie(G) \subset \VF(X)$  is locally finite. These Lie algebras are called {\it algebraic}.
\end{exa}
In contrast to the general situation we have the following finiteness result.
\begin{lem}\label{locallyfinite.lem}
Let $L \subset \VF(X)$ be a locally finite subspace. Then $L$ is finite dimensional.
\end{lem}
\begin{proof}
If $f_1,\ldots, f_m\in\OOO(X)$ is a set of generators, then every vector field  $\delta$ is determined by the values $\delta(f_1),\ldots,\delta(f_m)$. 
It follows that for any subspace $L \subset \VF(X)$ the linear map $L \to \OOO(X)^m$, $\delta \mapsto (\delta(f_1),\ldots,\delta(f_m))$ is injective.
 If $L$ is locally finite, then the image of this map is contained in the finite dimensional subspace $\bigoplus_{i=1}^m Lf_i$, and the claim follows.
\end{proof}

The fact that a vector field is determined by its values on a finite generating set of $\OOO(X)$ has the following important consequences.

\begin{lem}\label{locallyfiniteVF.lem}
Let $V \subset \OOO(X)$ be a finite dimensional  subspace generating  $\OOO(X)$,
and let $\delta \in \VF(X)$ be a vector field. If $V$ is $\delta$-invariant, 
then the following holds.
\be
\item $\delta$ is locally finite.
\item $\delta$ is semisimple if and only if $\delta|_V \in \gl(V)$ is semisimple.
\item $\delta$ is locally nilpotent if and only if $\delta|_V \in \gl(V)$ is nilpotent.
\item If $\delta = \delta_s + \delta_n$ is the Jordan decomposition in $\Ll(\OOO(X))$, then $\delta_s, \delta_n \in \VF(X)$.
\item $V$ is $\delta_s$- and $\delta_n$-invariant, and $\delta|_V = \delta_s|_V + \delta_n|_V \in \gl(V)$ is the usual Jordan decomposition in $\gl(V)$.
\ee
\end{lem} 
\begin{proof}
Denote by $V^{(n)} \subset \OOO(X)$ the linear subspace generated by the $n$-fold products of elements from $V$. 
Then the following is obvious: (a)~$V^{(n)}$ is $\delta$-invariant; (b) if $\delta|_V$ is semisimple resp. nilpotent, then so is $\delta|_{V^{(n)}}$. This proves (1)--(3). 
The claim (4)  follows from  [FK18, Proposition 7.6.1], and (5) follows  from 
(2) and (3).
\end{proof}

\begin{que}\label{Q2}
Let $L_{i}\subset \VF(X)$ ($i\in I$) be a family of locally finite Lie subalgebras, and 
denote by $L \subset \VF(X)$ the Lie subalgebra generated by the $L_i$. Is $L$ locally finite in case $L$ is finite dimensional?
\end{que}

In this direction we have the following consequence of  Theorem~A.

\begin{cor*}
Let  $L\subset \VF(X)$ be the Lie subalgebra generated by a family of locally finite Lie subalgebras $L_i\subset \VF(X)$, $i\in I$. 
Assume that each $L_i$ is algebraic, i.e. $L_i = \Lie G_i$ for a connected algebraic group $G_i\subset \Aut(X)$. 
Then $L$ is locally finite if and only if $L$ is finite dimensional. In this case the subgroup $G$ generated by the $G_i$ is algebraic and $L=\Lie G$. In particular, $L$ is algebraic.
\end{cor*}
\begin{proof}
If $L$ is locally finite, then it is finite dimensional, by Lemma~\ref{locallyfinite.lem}. If $L$ is finite dimensional, then the claim follows from Theorem~A.
\end{proof}
\begin{rem}
For a locally finite vector field $\delta \in \VF(X)$ there is a uniquely determined minimal algebraic group $H \subset \Aut(X)$ such that $\delta \in \Lie H$.  Moreover, $H$ is commutative and connected, and $H$ is a torus if $\delta$ is semisimple and 
$H \simeq \Ga$ if $\delta$ is locally nilpotent and non-zero, see \cite[Proposition~7.6.1]{FuKr2018On-the-geometry-of}.
\end{rem}
\ps
\subsection{The adjoint action on \tp{$\VF(X)$}{Vec(X)}}
For any vector field $\delta$ we have the {\it adjoint action\/} $\ad \delta$ on $\VF(X)$ defined in the usual way: 
$$
\ad \delta (\eta) := [\delta,\eta] = \delta\circ\eta - \eta \circ \delta.
$$
\begin{lem}\label{adjoint.lem}
Let $\delta \in \VF(X)$ be a locally finite vector field with Jordan decomposition $\delta =\delta_s +\delta_n$. 
Then $\ad\delta$ acting on $\VF(X)$ is locally finite, and $\ad\delta = \ad\delta_s + \ad\delta_n$ is the Jordan decomposition.
\end{lem}
\begin{proof}
Let $V \subset \OOO(X)$ be a finite dimensional $\delta$-invariant subspace which generates $\OOO(X)$. Then we have an inclusion 
\[
\VF(X) \into \Ll(V, \OOO(X)), \quad \mu\mapsto \mu|_V.
\]
Given $\mu \in \VF(X)$ we choose a finite dimensional $\delta$-invariant subspace $W$ which contains $\mu(V)$. 
In $\Ll(\OOO(X))$ we have 
\[
(\ad\delta)^{m} \mu =\sum_{i=0}^{m} (-1)^i \binom{m}{i}\delta^{m-i}\mu\,\delta^i.
\] 
Hence $(\ad\delta)^m \mu$ sends $V$ to $W$ for all $m\geq 0$, and so the linear span $\langle (\ad\delta)^m \mu \mid m\geq 0\rangle \subset \VF(X)$ is finite dimensional. 
This shows that $\ad\delta$ is locally finite.

If $\delta$ is locally nilpotent, then $(\delta|_V)^m = 0$ and $(\delta|_W)^m = 0$ for a suitable $m>0$.  The above formula implies that $(\ad\delta)^{2m-1} \mu = 0$, hence $\ad\delta$ is locally nilpotent.

Next, assume that $\delta$ is semisimple. Then we have eigenspace decompositions $V = \bigoplus_a V_a$ and $\OOO(X) = \bigoplus_b \OOO(X)_b$. As a consequence, we get a decomposition
$$
\Ll(V,\OOO(X)) = \bigoplus_{a,b} \Ll(V_a,\OOO(X)_b)
$$
into eigenspaces of $\ad \delta$, and so $\ad\delta$ is semisimple.

Finally, $\ad\delta_s$ and $\ad\delta_n$ commute, because $[\delta_s,\delta_n]=0$, and thus $\ad\delta = \ad\delta_s + \ad\delta_n$ is the Jordan decomposition.
\end{proof}
\ps
\subsection{Toral subalgebras}\label{toral.subsec}
\begin{defn}\label{toral.def}
A finite dimensional Lie subalgebra $S \subset \VF(X)$ is called {\it toral} if it consists of semisimple elements. 
It follows that $S$ is commutative \cite[8.1,~Lemma]{Hu1972Introduction-to-Li} and thus locally finite. 
In particular, we obtain a {\it weight decomposition\/} $\OOO(X) = \bigoplus_{\alpha\in S^*} \OOO(X)_\alpha$ where
$$
\OOO(X)_\alpha := \{f \in \OOO(X) \mid \tau f = \alpha(\tau) \cdot f \text{ for } \tau \in S\}.
$$
\end{defn}
If $T \subset \Aut(X)$ is a torus, then $\Lie T \subset \VF(X)$ is a toral subalgebra, 
but it is not true that every toral subalgebra is of this form. Indeed, if $T \subset \Aut(X)$ is a two dimensional torus, 
then the one dimensional tori $S \subset T$ are the kernels of certain characters $\chi \in X(T)\simeq \ZZ^2$ 
whereas every one dimensional subspace of $\Lie T$ is a toral subalgebra. The best we can get is the following result.
\begin{lem}\label{toral.lem}
For a toral Lie subalgebra $S \subset \VF(X)$ there exists a uniquely defined smallest torus $T \subset \Aut(X)$ such that $\Lie T \supset S$. 
Moreover, if a subspace $M \subset \VF(X)$ is invariant under the adjoint action of $S$, then $M$ is invariant under $T$.
\end{lem}
\begin{proof} 
(a)
One easily verifies that the weight decomposition of $\OOO(X)$ has the following properties, cf. 
\cite[Proposition~7.6.1 
and its proof]{FuKr2018On-the-geometry-of}:
\be
\item[$\bullet$]
$\OOO(X)_\alpha \cdot \OOO(X)_\beta \subset \OOO(X)_{\alpha+\beta}$;
\item[$\bullet$]
the set of weights $\Lambda:=\{\gamma\mid \OOO(X)_\gamma\neq 0\} \subset S^*$ generates a free subgroup $\ZZ\Lambda = \bigoplus_{i=1}^k \ZZ \mu_i$.
\ee 
This defines a faithful action of the torus $T:=(\Gm)^k$ on $\OOO(X)$ in the following way: For $\alpha=\sum_i n_i\mu_i \in \Lambda$, 
$f \in \OOO(X)_\alpha$ and $t = (t_1,\ldots,t_k) \in T$ we put
$$
t \, f := t_1^{n_1} \cdots t_k^{n_k} \cdot f.
$$
The corresponding action of $c=(c_1,\ldots,c_k) \in \Lie T=\kk^m$ is then given by the vector field $\delta_c$ where 
$$
\delta_c|_{\OOO(X)_\alpha} = \text{scalar multiplication with }\textstyle \sum_i n_i c_i \text{ for }\alpha = \sum_i n_i\mu_i \in \Lambda. 
$$ 
It follows that $\Lie T$ contains $S$ and that every torus with this property contains $T$, see 
\cite[proof of Proposition~7.6.1(2)]{FuKr2018On-the-geometry-of}.
\ps
(b)
With respect to the adjoint action of $S$ on $\VF(X)$ we also have a weight space decomposition 
$\VF(X) = \bigoplus_{\gamma \in  S^*} \VF(X)_\gamma$  where
$$
\VF(X)_\gamma :=\{\delta \in \VF(X) \mid (\ad\tau)\delta = \gamma(\tau)\cdot\delta \text{ for }\tau\in S\},
$$
see Lemma~\ref{adjoint.lem}. Equivalently, we have
$$
\VF(X)_\gamma :=\{\delta \in \VF(X) \mid
\delta(\OOO(X)_\alpha) \subset \OOO(X)_{\alpha+\gamma}
\text{ for }\alpha\in S^*\}.
$$
It follows that if $\VF(X)_\gamma \neq 0$, then $\gamma \in \ZZ\Lambda$. In fact, if $\delta \in\VF(X)_\gamma$, $\delta\neq 0$, then there is an $\alpha\in\Lambda$ and a non-zero $f \in \OOO(X)_\alpha$ such that $\delta f \neq 0$. Thus $\alpha$ and $\alpha+\gamma$ belong to $\Lambda$, and so $\gamma \in \ZZ\Lambda$. 
Hence, the subspaces $\VF(X)_\gamma$ are weight spaces of the action of $T$ on $\VF(X)$, and thus every $S$-invariant subspace of $\VF(X)$ is invariant under $T$.
\end{proof}
\ps
\subsection{Integration of Lie algebras}\label{integrationLA.subsec}
The first part of the following result is due to \name{Cohen} and \name{Draisma}, see 
\cite[Theorem 1]{CoDr2003From-Lie-algebras-}.  

\begin{thmE}
Let $L\subset \VF(X)$ be a Lie subalgebra. Then $L \subset \Lie G$ for
some algebraic subgroup $G \subset \Aut(X)$ if and only if
$L$ is locally finite. If $L$ is locally finite and generated by locally nilpotent elements, then $L$ is algebraic, i.e. there is an algebraic group $G\subset \Aut(X)$ such that $L = \Lie G$.
\end{thmE}
The proof will be given in Section~\ref{integration.subsec}, see Theorem~\ref{CD.thm}.
A weaker form of Question~\ref{Q2} is the following.
\begin{que} 
Let $\xi,\eta \in \VF(X)$ be locally finite vector fields. If the Lie subalgebra $L:=\langle \xi,\eta \rangle_{\text{\tiny Lie}}$ 
generated by $\xi$ and $\eta$ is finite dimensional, does it follow that $L$ is locally finite? 
\end{que}
From Theorem~E we get the following result in this direction.
\begin{cor*}
Let $\{\eta_i\mid i\in I\}$ be a family of locally nilpotent vector fields, and denote by $L:=\langle \eta_i \mid i \in I\rangle_{\text{\tiny Lie}}$ 
the Lie subalgebra generated by the $\eta_i$. If $L$ is finite dimensional, then $L$ is locally finite and algebraic.
\end{cor*}
\ps
\subsection{A free subgroup of \tp{$\Aut(\Atwo)$}{Aut(A2)} and its closure} \label{sec:free-sbgrp}
In Section~\ref{sec:example} we provide an example of a 
nonabelian free subgroup  $F
\simeq \mathbb{F}_2$ of $\Aut(\Atwo)$ generated by two locally finite elements whose product is not locally finite. 
In particular, $F$ is not locally finite. We compute its closure  $\F:=\overline{F}\subset \Aut(\Atwo)$ and describe the 
Lie algebra $\Lie\F$.  It occurs that $\F$ coincides with the double centralizer of $F$
(Lemma~\ref{example.lem}) and is a free product of two abelian nested unipotent ind-subgroups $\J$ and $\J^-$ (Theorem~\ref{example.thm}(3)). 
Furthermore, 
any algebraic subgroup $G$ of $\F$ is abelian and unipotent and is conjugate to a subgroup of $\J$ or $\J^-$
(Theorem~\ref{example.thm}(4)).

The construction shows that $\F$ contains two algebraic subgroups $U$ and $V$, both isomorphic to $\Ga$, with the following properties: 
$$
\langle U, V \rangle = U*V, \quad
\F = \overline{\langle U,V \rangle} \quad \text{and} 
\quad \Lie \F = \langle \Lie U, \Lie V\rangle_{\text{\tiny Lie}},
$$
see Lemma~\ref{example-lie.lem}.
This is an instance where Question~\ref{Q1} has a positive answer.
\par\medskip
Yet another instance is the following. Consider the ind-group $\G :=\Aut_{\KKK}(\KKK[x,y])$ where $\KKK:=\kk[z]$, 
and let $\G^t\subset \G$ be the subgroup consisting of tame automorphisms 
of $\kk[x,y,z]$. Then $\G$  is connected and closed in $\Aut_\kk(\kk[x,y,z])$, $\G^t\subsetneqq\G$ is closed and 
$\Lie\G^t=\Lie\G$, see \cite[Theorem~17.3.1]{FuKr2018On-the-geometry-of}. The group $\G^t$ is generated by the affine and the de Jonqui\`eres
transformations which fix $z$ (cf. \cite[Remark~17.3.10]{FuKr2018On-the-geometry-of}). 
Since the de Jonqui\`eres group is nested (Remark~\ref{rem1}), it is alge\-braically generated (Remark~\ref{AG.rems}(1)), and so $\G^t$ is algebraically generated. 
By the  \name{Umirbaev-Shestakov} Theorem, the
one dimensional  unipotent subgroup $U$ of $\G$ corresponding to the locally nilpotent Nagata derivation is not contained in $\G^t$. 
However, since the algebraically generated group $\H:=\langle \G^t, U\rangle$ is contained in $\G$ and $\Lie U\subset \Lie\G =\Lie\G^t$ we have 
\[
\Lie\overline{\H} = L_{\H}:= \langle \Lie\G^t, \Lie U\rangle = \Lie\G.
\] 
\ps
\subsection{Ind-group actions and path-connected subsets}\label{path-connected.subsec}
The theory of algebraic group actions on affine varieties can be generalized, to some extent, to ind-groups acting on affine ind-varieties,
see \cite{FuKr2018On-the-geometry-of}. If $\G$ is an ind-group, then a $\G$-variety
$X$ is a variety with a $\G$-action such that the map $\G\times X \to X$
is a morphism of ind-varieties. 

A subset $M \subset \X$ of an ind-variety $\X$ is called {\it path-connected} if for any two points $x,y \in M$ 
there is an irreducible variety $Y$ and a morphism $\phi\colon Y \to M$ such that $x,y \in \phi(Y) \subset \X$, 
cf. \cite[p.~26]{Ra1964A-note-on-automorp}. 
One can always assume that $Y$ is an irreducible curve, because any two points in an irreducible variety are contained in an irreducible curve, 
cf. \cite[Section~1.6]{FuKr2018On-the-geometry-of}. A connected ind-variety is not necessarily path-connected (\cite[Example~1.6.5]{FuKr2018On-the-geometry-of}), but a connected ind-group is path-connected (\cite[Proposition~2.2.1]{FuKr2018On-the-geometry-of}).

Note that a path-connected subset of an algebraic variety is not necessarily closed; take $\CC^2$ and remove a ball around the origin. However, we will see that a path-connected subgroup of an algebraic group is closed, cf. Corollary~\ref{path-connected-subgroup.cor}.

In contrast to the finite-dimensional case, a strict closed subgroup of a path-connected ind-group may have the same 
Lie algebra, see \cite[Theorem 17.3.1]{FuKr2018On-the-geometry-of} and the example in Section~\ref{sec:free-sbgrp}. 
This disproves \cite[Theorems 1 and 2]{Sh1981On-some-infinite-d}. 
However, a homomorphism of a path-connected ind-group to an ind-group is uniquely determined by the induced homomorphism 
of their Lie algebras, see \cite[Proposition~7.4.7]{FuKr2018On-the-geometry-of}. 

In the last section~\ref{canonical-LA.sec} we define a canonical Lie subalgebra $L_\G \subset \Lie \G$ for an ind-group $\G$ and show that for a surjective homomorphism $\G \to \H$ the corresponding homomorphism $L_\G \to L_\H$ is also surjective. We do not know if this holds for $\Lie\G \to \Lie\H$. We also study subgroups $\H \subset \G$ of finite codimension and show that under additional assumptions $L_\H$ has finite codimension in $L_\G$.

\par\bigskip
\section{Algebraically generated subgroups}\label{sec:AGG}
\subsection{Orbits of algebraically generated subgroups}\label{subsec:orbits-of-agg}
Let $\G=\bigcup_{k}\G_{k}$ be an ind-group.
The following concept was introduced in \cite{ArFlKa2013Flexible-varieties}.
\begin{defn}\label{AGgroup.def}
A subgroup $G \subset \G$ is called {\it algebraically generated} if there is
a family $(G_{i})_{i\in I}$ of connected algebraic subgroups $G_{i}\subset \G$
generating $G$ as an abstract group: $G=\langle G_{i}\mid i\in I\rangle$.
\end{defn}
It is clear that an algebraically generated subgroup is path-connected (section~\ref{path-connected.subsec}), hence connected. The following result
can be found in  \cite[Prop.\ 1.2 and 1.3]{ArFlKa2013Flexible-varieties}. 
For the reader's convenience we provide a short argument. 
\begin{prop}\label{orbit-agg.prop}
Let $G=\langle G_{i}\mid i\in I\rangle \subset \Aut(X)$ be an algebraically generated subgroup.
For every $x\in X$ the $G$-orbit $G x \subset X$ is open in its closure $\overline{G x}$.
Moreover, there is a finite sequence $j_{1},\ldots, j_{n}\in I$ such that
$G x = G_{j_{1}}\cdots G_{j_{n}} x$ for all $x\in X$.
\end{prop}
\begin{proof}
(a) Choose a sequence $i_{1},\ldots, i_{m}$ such that $\dim
\overline{G_{i_{1}}\cdots G_{i_{m}} x}$ is maximal. Since
$\overline{G_{i_{1}}\cdots G_{i_{m}} x}$ is irreducible, it
follows that $G_{j_{1}}\cdots G_{j_{\ell}} x \subset
\overline{G_{i_{1}}\cdots G_{i_{m}} x}$ for every sequence
$j_{1},\ldots,j_{\ell}\in I$. Hence $\overline{G_{i_{1}}\cdots
G_{i_{m}} x} = \overline{Gx}$, and there is an open dense set $U
\subset \overline{Gx}$ contained in $G_{i_{1}}\cdots G_{i_{m}} x$.
This implies that $Gx = GU$ is open in $\overline{Gx}$. 
\ps
(b) For the second claim we first remark that for every $x\in X$ with
$\dim Gx$ maximal there is an open set $U_{x}\subset X$ containing
$x$ and a sequence $\ii_{x}=(i_{1},\ldots,i_{m})$, depending on
$x$, such that $\overline{G_{i_{1}}\cdots G_{i_{m}} y} =
\overline{Gy}$ for all $y\in U_{x}$. Indeed, if $d:=\dim Gx$ is
maximal and $Gx \subset \overline{G_{i_{1}}\cdots G_{i_{m}} x}$,
then $\dim \overline{G_{i_{1}}\cdots G_{i_{m}} y} = d$ for all $y$
in an open neighborhood of $x$. Now the union $U:=\bigcup_{x}
U_{x}$ is a $G$-invariant open set, hence covered by finitely many
$U_{x_{j}}$. By joining the finitely many sequences $\ii_{x_{j}}$
corresponding to the points $x_{j}$, we obtain a sequence
$(j_{1},\ldots,j_{\ell})$ such that $\overline{G_{j_{1}}\cdots
G_{j_{\ell}}y} = \overline{Gy}$ for all $y\in U$. A standard
argument implies that $G_{j_{\ell}} G_{j_{\ell-1}}\cdots
G_{j_{2}}G_{j_{1}}G_{j_{2}}\cdots G_{j_{\ell}}y = Gy$ for all
$y\in U$. Since $X \setminus U$ is a closed $G$-invariant subvariety
of smaller dimension, the claim follows by induction.
\end{proof}

\begin{rem} A similar result holds for the action of a connected ind-group $\G$ on an affine variety $X$,
see \cite[Proposition~7.1.2]{FuKr2018On-the-geometry-of}.
\end{rem}

\ps
\subsection{The multiplication map}\label{mult-map.subsec}
In the following lemma we use the adjoint action $\Ad$ of $\G$ on its Lie algebra $\Lie\G$, cf. \cite{FuKr2018On-the-geometry-of}. Moreover, for $g\in \G$, we denote by $\rho_g\colon \G \simto \G$ the right multiplication with  $g$.
\begin{lem}\label{crucial.lem1}
Let $G_1,\ldots,G_m \subset \G$ be connected algebraic subgroups
and consider the ``multiplication map'' $\phi\colon G_1 \times \cdots \times G_m \to \G$.
Fix $g_i \in G_i$, $i=1,\ldots,m$, and set $g:=g_1\cdots g_m$.
Then the image of the composition
\[
d(\rho_{g^{-1}})_g \circ d\phi_{(g_1,\ldots,g_m)}\colon T_{g_1} G_1 \oplus \cdots \oplus T_{g_m} G_m \to \Lie\G
\] 
is given by $\sum_i \Ad h_i (\Lie G_i)$ where $h_1=e$ and $h_i:=g_1\cdots g_{i-1}$ for $i>1$.
\end{lem}
\begin{proof} 
The restriction of the composition $\rho_{g^{-1}}\circ \phi$ to 
$$
\{g_1\}\times\cdots\times\{g_{i-1}\}\times G_i \times\{g_{i+1}\}\times\cdots\times\{g_n\}
\subset G_1 \times\cdots\times G_n
$$
is given by 
$$
\phi_i\colon G_i \to \G,\quad  h \mapsto g_1\cdots g_{i-1} h g_{i+1}\cdots g_n g^{-1} = h_i (hg_i^{-1}) h_i^{-1}
$$
i.e. $\phi_i = \Int h_i \circ \rho_{g_i^{-1}}$, and so the image of $T_{g_i}G_i$ under $d(\phi_i)_{g_i}$ is equal to $\Ad h_i (\Lie G_i)$.
\end{proof}

In case of an automorphism group $\G = \Aut(X)$ we can reformulate the lemma above in terms of vector fields. Recall that the action of $\Aut(X)$ on the vector fields is given by $\Ad g(\delta)(gx) = d\mu_g \, \delta(x)$, see section~\ref{VF.subsec}.

If $L \subset \VF(X)$ is a subspace, then $L(x) \subset T_xX$ denotes the image of $L$ under the evaluation map $\ev_x\colon \VF(X) \to T_x X$. 
\begin{lem}\label{crucial.lem2}
Let $G_1,\ldots,G_m \subset \Aut(X)$ be connected algebraic subgroups. Set $L_i:=\Lie G_i \subset \VF(X)$ and consider the ``orbit map''
\[
\theta_x \colon G_1 \times \cdots \times G_m \to X, \quad \bg:=(g_1,\ldots,g_m)\mapsto \bg x:=g_1\cdots g_m x.
\]
Then the image of the differential $(d\theta_x)_{\bg} \colon T_{g_1} G_1\oplus\cdots\oplus T_{g_m} G_m \to T_{\bg x}X$ is given by 
$\sum_i d\mu_{h_i} L_i(g_{i}\cdots g_m x)$.
\end{lem}
\begin{proof}
By Lemma~\ref{crucial.lem1} above the image of the differential $(d\theta_x)_\bg$ is given by $\sum_i \Ad h_i (L_i) (\bg x)$. Since $\Ad h_i (L_i) (\bg x) = d\mu_{h_i} L_i(g_i\cdots g_m x)$ the claim follows.
\end{proof}

\ps
\subsection{Lie algebras of algebraically generated groups}\label{liealg-agg.subsec}
For an algebraically generated subgroup $G=\langle G_{i}\mid i\in I\rangle \subset \G$ of the ind-group $\G$ we define
$$
L(G):=\langle \Lie G_{i}\mid i\in I\rangle_{\text{\rm Lie}}\subset \Lie\G
$$ 
to be the Lie subalgebra
generated by the $\Lie G_{i}$. If $Z \subset \G$ is an algebraic subset and $g \in Z$ we can use the right multiplication with $g^{-1}$ to send $T_g Z$ into $T_e \G = \Lie \G$: $d\rho_{g^{-1}}\colon T_g Z \into \Lie \G$.
Note that for an algebraic subgroup $G \subset\G$ this image of $T_gG$ is equal to $\Lie G$ for any $g \in G$.
\begin{thm}\label{main1.thm}
Let $G = \langle G_{i}\mid i\in I \rangle \subset \G$ be an algebraically generated subgroup.
\be
\item\label{item1}
$L(G)$ is stable under the adjoint actions of $G$ and of $\overline{G} \subset \G$. 
In particular, $L(G)$ is an ideal in $\Lie\overline{G}$.
\ee
Assume in addition that the index set $I$ is countable and that the base field $\kk$ is uncountable.
\be
\item[(2)]
If $Z \subset \G$ is an algebraic subset contained in $G$, then the image of $T_gZ$ under $d(\rho_{g^{-1}})_g$ is contained in $L(G)$ for any $g$ in an open dense subset of $Z$.
\item[(3)]
If $H \subset \G$ is an algebraic subgroup contained in $G$, then $\Lie H \subset L(G)$. In particular, $L(G)$ 
depends only on $G$ and not on the generating subgroups $G_{i}$.
\ee
\end{thm}
\begin{rem}
We do not know if the additional assumptions for the statements (2) and (3) are necessary.
\end{rem}
\begin{proof}
(1)
It suffices to show that  $L(G)$ is invariant under all the $G_{i}$. Since $G_{i}$ is connected, we are reduced to prove that $L(G)$ is invariant under the adjoint action of $\Lie G_i$ (see \cite[Proposition 6.3.4]{FuKr2018On-the-geometry-of}) which holds by construction.
\ps
(2)
We first show that $Z \subset G_{i_1}\cdots G_{i_m}$ for a suitable sequence $(i_1,\ldots,i_m)$. We can clearly assume that $Z$ is irreducible. Since the index set $I$ is countable the group $G$ is the union of countably many products $G_{i_1}\cdots G_{i_m}$, and so $Z$ is the union of countably many constructible subsets $Z\cap G_{i_1}\cdots G_{i_m}$. Since the base field $\kk$ is uncountable we get $Z = Z\cap G_{i_1}\cdots G_{i_m}\subset G_{i_1}\cdots G_{i_m}$ for a suitable sequence $(i_1,\ldots,i_m)$ (\cite[Lemma~1.3.1]{FuKr2018On-the-geometry-of}).

Now consider the multiplication map $\phi\colon G_{i_1}\times \cdots\times G_{i_m} \to \G$. The induced morphism $\phi'\colon \phi^{-1}(Z) \to Z$ is smooth on a non-empty open set $U \subset \phi^{-1}(Z)$ where we can assume that $\phi(U) \subset Z_\reg$. Since $d\phi_u\colon T_u\phi^{-1}(Z) \to T_{\phi(u)} Z$ is surjective for all $u \in U$ it follows from Lemma~\ref{crucial.lem1} that the image of the tangent space $T_{\phi(u)} Z$ in $\Lie\G$ is contained in 
$\sum_k \Ad h_{i_k}(\Lie G_{i_k}) \subset L(G)$ where the latter inclusion follows from (1).
\ps
(3)
This is an immediate consequence of (2).
\end{proof}

Now consider the case where $G=\langle G_{i}\mid i\in I\rangle \subset \Aut(X)$ and so $L(G) \subset \Lie\Aut(X) \subset \VF(X)$.
If $Y \subset X$ is a $G$-stable closed subset, then
$\Lie G_{i}|_{Y}\subset \VF(Y)$ for all $i$ and so $L(G)|_{Y}\subset \VF(Y)$.
Choosing $Y:=\overline{Gx}$ we get  $L(G)(x) \subset T_{x}Gx$ for any $x\in X$.

\begin{prop}
We have  $T_{x}Gx = L(G)(x)$ for all $x\in X$. In particular, any vector field in $L(G)$
is tangent to the $G$-orbits.
\end{prop}
\begin{proof}
We already know that $L(G)(x) \subset T_x Gx$. Now we choose a surjective ``orbit map'' 
\[
\theta_x\colon G_{j_1}\times \cdots \times G_{j_n} \to Gx \subset X,
\]
see Proposition~\ref{orbit-agg.prop}. It follows that for a suitable $\bg = (g_1,\ldots,g_n)$ the differential 
\[
(d\theta_x)_{\bg} \colon T_{g_1}G_{j_1} \oplus \cdots \oplus T_{g_n} G_{j_n} \to T_{\bg x} Gx
\]
is also surjective. Thus we get $T_{\bg x} Gx = \sum_j d\mu_{h_j} L_j(g_{j}\cdots g_n x)$,
by Lemma~\ref{crucial.lem2}. 
By Theorem~\ref{main1.thm}(1) $L(G)$ is invariant under the action of $G$, and so
\[
d\mu_{h_j} L_j(g_{j}\cdots g_n x) \subset d\mu_{h_j} L(G)(g_{j}\cdots g_n x) =  L(G)(\bg x),
\] 
and the claim follows.
\end{proof}

\ps
\subsection{Ramanujam's alternative}\label{ramanujam}
The following result is due to \name{Ramanujam} (\cite[Theorem on p.~26]{Ra1964A-note-on-automorp}).
\begin{prop}\label{Ramanujam.prop}
Let $G \subset \G$ be a path-connected subgroup of an ind-group $\G$. Then one of the following holds:
\be
\item[{\rm (i)}] $G$ is an algebraic subgroup;
\item[{\rm (ii)}] $G$ contains algebraic subsets of arbitrary large dimension.
\ee
\end{prop}

\begin{proof}
Assume that we are not in case (ii), and let $Y \subset G$ be an irreducible algebraic subset
of maximal dimension $d$. We may assume that $e\in Y$. We claim that $G \subset
\overline{Y}=\overline{G}$. 

Since $G$ is path-connected we can find, for any $g \in G$, an irreducible variety $C_g$ and a morphism 
$\phi_g\colon C_g \to \G$ such that $\phi_g(C_g) \subset G$ and $e,g \in \phi_g(C_g)$. Clearly, 
$\phi_g(C_g)\cdot Y \subseteq G$, and the closure $\overline{\phi_g(C_g)\cdot Y} \subset \G$ is an irreducible subvariety which contains $\overline{Y}$ and $g$. The constructible set $\phi_g(C_g)\cdot Y$ contains a subset $U_g$ which is open and dense in its closure $\overline{\phi_g(C_g)\cdot Y}$. Thus $\dim U_g \leq d$, and so $\overline{U_g} = \overline{\phi_g(C_g)\cdot Y} = \overline{Y}$. Hence, $G \subset \overline{Y} = \overline{G}$ as claimed.

As a consequence, $\overline{G}\subset \G$ is an algebraic subgroup. Since $Y\subset G$
is open and dense in $\overline{G}$ it follows that $\overline{G} = Y^{-1} \cdot Y \subset G$,
and so (i) holds.
\end{proof}
\begin{cor}\label{path-connected-subgroup.cor}
A path-connected subgroup of an algebraic group is a closed subgroup.
\end{cor}

\ps
\subsection{The finite dimensional case}\label{finitedim.subsec}
The next result shows that an algebraically generated subgroup of an ind-group $\G$ is an algebraic group if and only if the Lie algebra $L(G)$ is finite dimensional. 
It also proves Theorem~A from Section~\ref{alg-gen-groups:sec} of the introduction.

\begin{thm}\label{main2.thm}
Let $G = \langle G_{i}\mid i\in I \rangle \subset \G$ be an algebraically generated group. Then the following two statements are equivalent.
\be
\item[(i)]
$G$ is an algebraic subgroup of $\G$;
\item[(ii)] The Lie algebra $L(G)$ is finite dimensional.
\ee
In this case $\Lie G = L(G)$. In particular, $L(G)$ is locally finite and algebraic.
\end{thm}

\begin{proof}
(a) Assume that $G$ is an algebraic group. Since $G_i \subset G$ is a closed subgroup for any $i$, we have $\Lie G_i \subset \Lie G$, and so $L(G)\subset \Lie G$ is finite dimensional. 
Moreover, there is a finite sequence $(i_1,\ldots,i_m)$ such that the multiplication map $\phi\colon G_{i_1}\times\cdot\times G_{i_m} \to G$ is surjective. Then Lemma~\ref{crucial.lem1} applied to $g_1 = g_2 = \cdots = g_m = e$ shows that $\Lie G \subset \sum_i \Lie G_i$, and so $L(G) = \Lie G$.
\par\smallskip
(b) Now assume that $L(G)$ is finite dimensional, and consider the multiplication map $\phi\colon G_{i_1}\times\cdot\times G_{i_m} \to \G$. There is an open dense subset $U \subset G_{i_1}\times\cdot\times G_{i_m}$ such that $\phi(U) \subset \overline{G_{i_1}\cdots G_{i_m}}\subset \G$ is open and contained in 
$(\overline{G_{i_1}\cdots G_{i_m}})_\reg$, and that $\phi|_U \colon U \to \phi(U)$ is smooth. It follows from Lemma~\ref{crucial.lem1} that $T_g \phi(U) \subset L(G)$ for all $g \in \phi(U)$ and so $\dim \overline{G_{i_1}\cdots G_{i_m}} \leq \dim L(G)$. Hence there is a sequence $(i_1,\ldots,i_m)$ such that $\dim\overline{G_{i_1}\cdots G_{i_m}}$ is maximal which implies that all $G_i$ are contained in $\overline{G_{i_1}\cdots G_{i_m}}$, and so $G$ is an algebraic subgroup.
\end{proof}
\ps
\section{Integration of vector fields}\label{integration.sec}
\subsection{Integration of locally finite Lie algebras}\label{integration.subsec}
The next result is our Theorem~E from the introduction, see Section~\ref{integrationLA.subsec}.
\begin{thm}\label{CD.thm}
Let $L\subset \VF(X)$ be a Lie subalgebra. 
\be
\item
If $L$ is finite dimensional and generated by locally nilpotent vector fields, then $L$ is algebraic, i.e. there is a connected algebraic subgroup $G \subset \Aut(X)$ such that $L = \Lie G$.
\item
If $L$ is locally finite, then there exists a unique minimal connected algebraic subgroup 
$G \subset \Aut(X)$ such that $L \subset \Lie G$.
\ee
\end{thm}
\begin{proof}[\protect Proof of (1)]
Assume that $L$ is finite dimensional and generated by locally nilpotent vector fields $\eta_1,\eta_2, \ldots,\eta_m$. 
By \cite[Proposition~7.6.1]{FuKr2018On-the-geometry-of} there are (uniquely determined) algebraic subgroups $U_i \subset \Aut(X)$, 
$U_i \simeq \Ga$, such that $\Lie U_i = \kk \eta_i$. Now Theorem~\ref{main2.thm}  implies that $G:=\langle U_i \mid i=1,\ldots,m \rangle \subset \Aut(X)$ 
is an algebraic group with $\Lie G = L$.
\end{proof}

\begin{exa}\label{semisimple.exa}
Assume that $L \subset \VF(X)$ is a locally finite semisimple Lie subalgebra. Then there is a semisimple group $G \subset \Aut(X)$ such that $L = \xi(\Lie G)$. In particular, $L$ is algebraic.

In fact, the semisimple Lie algebra $L$ is generated by the root subspaces with respect to some Cartan subalgebra 
(see \cite[Proposition 8.4(f)]{Hu1972Introduction-to-Li}), hence by nilpotent elements. Lemma~\ref{locallyfiniteVF.lem}(3) shows that these elements are locally nilpotent in $\VF(X)$, and thus $L$ is generated by locally nilpotent vector fields.
\end{exa}
The proof of (2) needs some preparation. It will be given at the end of section~\ref{saturated.sec}. 
Let  $L= L_0 \ltimes R$ be the Levi decomposition  where $L_0$ is semisimple and $R:=\rad L$ is the solvable radical (see \cite[10.1.6, pp. 305--306]{Pr2007Lie-groups}. 
The example above shows that $L_0 = \Lie G_0$ for a semisimple subgroup $G_0\subset \Aut(X)$.
It remains to find a solvable subgroup $Q \subset \Aut(X)$ normalized by $G_0$ such that $\Lie Q \supset R$. Then the product $G:= G_0 Q$ has the required property.
\begin{lem}\label{solvable.lem}
Let $M \subset \VF(X)$ be a locally finite solvable Lie subalgebra. Assume that $M$ contains with every element $\mu$ 
its semisimple and locally nilpotent parts $\mu_s$ and $\mu_n$. 
Then the locally nilpotent elements from $M$ form a Lie subalgebra $M_n$, and $M = S \ltimes M_n$ where $S$ is any maximal toral subalgebra of $M$ (see Definition~\ref{toral.def}).
\end{lem}
\begin{proof}
Since $M$ is locally finite we may assume that $M \subset \gl(V)$ where $V \subset \OOO(X)$ is a finite dimensional  subspace generating $\OOO(X)$.
Since $M$ is solvable there is a basis of $V$ such that $M$ is contained in the upper triangular matrices 
$\bb_d \subset \gl_d$ where $d:=\dim V$ (Lie's Theorem). Clearly, the nilpotent elements from $M$ belong to the subalgebra 
$\nn_d \subset \bb_d$ of upper triangular matrices with zeros along the diagonal, and so $M_n:=M \cap \nn_d$ is the ideal of $M$ consisting of the nilpotent elements.
 By  Lemma~\ref{locallyfiniteVF.lem} every nilpotent element of $M$ is also locally nilpotent viewed as a vector field on $X$, and vice versa. If $M=M_n$ then $S=0$ 
 and $M=S\ltimes M_n$, as desired. 
 
In the general case,
choose a maximal toral subalgebra $S \subset M$, and consider the corresponding weight space decomposition 
$M = \bigoplus_{\alpha\in S^*} M_\alpha$ where $S^*$ is the group of characters of $S$ and 
$$
M_\alpha := \{\mu \in M \mid \ad \tau (\mu)=[\tau , \mu] = \alpha(\tau) \cdot \mu \text{ for }\tau \in S\}.
$$
Similarly, we get a decomposition $\bb_d = \bigoplus_{\alpha\in S^*} (\bb_d)_\alpha$. 
Using the identity $[s,a \cdot b]=[s,a]\cdot b + a \cdot [s,b]$ we get
$(\bb_d)_\alpha \cdot (\bb_d)_\beta \subset (\bb_d)_{\alpha + \beta}$ which implies that $M_\alpha$ consists of nilpotent elements for $\alpha\neq 0$. 
Hence, $\bigoplus_{\alpha\neq 0} M_\alpha \subset \nn_d$,
while $M_0=\cent_M(S)\supset S$.

Let $\mu \in M_0$ have Jordan decomposition $\mu = \mu_s + \mu_n$. Since $S$ commutes with $\mu$ 
it also commutes with $\mu_s$ and $\mu_n$, and so $\mu_s \in M_0$ and $\mu_n \in M_0\cap \nn_d$.  Since $S$ is a maximal toral subalgebra, we have $\mu_s\in S$, and so
$M_0 = S \oplus (M_0\cap\nn_d)$. As a consequence we see that  $M_n=(M_0\cap \nn_d)\oplus \bigoplus_{\alpha\neq 0} M_\alpha \subseteq \nn_d$ and $M = S \ltimes M_n$.
\end{proof}

From the first part of the proof we get the following result.
\begin{cor}\label{cor:lnd} 
If a locally finite solvable Lie subalgebra $M\subset \VF(X)$ is generated by locally nilpotent elements, then it consists of locally nilpotent elements.
\end{cor}
\ps
\subsection{Jordan-saturated subspaces}\label{saturated.sec}
The assumption in Lemma~\ref{solvable.lem} above leads to the following definition.
A locally finite subspace $W \subset \VF(X)$ is called {\it Jordan-saturated\/} (shortly {\it $J$-saturated}) if it contains 
with every element $\eta$ the semisimple and the locally nilpotent parts $\eta_s$ and $\eta_n$. 
Clearly, for any algebraic subgroup $G \subset \Aut(X)$  the Lie algebra $\Lie G\subset \VF(X)$ is $J$-saturated.
\begin{lem}\label{J-saturation.lem}
Let $R \subset \VF(X)$ be a locally finite and solvable Lie subalgebra. Then the smallest $J$-saturated Lie subalgebra $\tilde R \subset \VF(X)$ 
containing $R$ has the following properties:
\be
\item 
$\tilde R$ is locally finite and solvable.
\item
If a subspace $V \subset \OOO(X)$ is $R$-invariant, then it is invariant under $\tilde R$.
\item
If an automorphisms $\phi\in \Aut(X)$ normalizes $R$, then it normalizes $\tilde R$.
\ee
\end{lem}
Notice that any algebraic Lie subalgebra of $\VF(X)$ which contains $R$  also contains $\tilde R$. 
\begin{proof}
As before, we fix a finite dimensional $R$-invariant subspace $V \subset \OOO(X)$ which generates $\OOO(X)$. 
Choosing a suitable basis of $V$ we may assume that the image of $R$ in $\gl(V)=\gl_d$ is contained in the upper triangular matrices $\bb_d$. 

If $\eta = \eta_s + \eta_n$ is the Jordan-decomposition of an element $\eta \in R$, then every finite dimensional $R$-invariant subspace 
$V \subset \OOO(X)$ is invariant under $\eta_s$ and $\eta_n$, hence invariant under the Lie subalgebra $R_1$ generated by $R+\kk \eta_n$. Thus $R_1$ is locally finite. 
The Jordan-decomposition of $\eta$ carries over to the image of $\eta$ in $\gl(V) = \gl_d$. Hence the images of $\eta_s$ and $\eta_n$ also belong to $\bb_d$ 
which shows that the image of $R_1$ is contained in $\bb_d$. 

This procedure of adding the locally nilpotent part and forming the Lie algebra has to end with a locally finite and solvable Lie subalgebra 
$\tilde R$ containing $R$ with the properties (1) and (2).

Finally, assume  that $\phi\in\Aut(X)$ stabilizes $R$, and let $\eta\in R$. Then $\phi(\eta) = \phi(\eta_s) + \phi(\eta_n)$ is the Jordan 
decomposition of $\phi(\eta)\in R$, hence $\phi(\eta_s)$ and $\phi(\eta_n)$ belong to $\tilde R$. 
Now it follows from the construction that $\tilde R$ is invariant under $\phi$.
\end{proof}
\begin{proof}[\protect Proof of Theorem~\ref{CD.thm}(2)]
Let $L = L_0 \ltimes R$ be the Levi decomposition where $L_0$ is semisimple and $R$ is the solvable radical of $L$. 
By part (1) and Example~\ref{semisimple.exa} there is a semisimple group $G_0 \subset \Aut(X)$ with $\Lie G_0 = L_0$. 
Clearly, $R$ is $G_0$-invariant, as well as $\tilde R$, the $J$-saturation of $R$, see Lemma~\ref{J-saturation.lem}. 
By Lemma~\ref{solvable.lem} we have $\tilde R = S \ltimes \tilde R_n$ where $S$ is toral and $\tilde R_n$ is solvable and consists of locally nilpotent elements. 
Thus, again by part (1), there is a unipotent group $U \subset \Aut(X)$ such that $\Lie U = \tilde R_n$. 

It follows from Lemma~\ref{toral.lem} that there is a well-defined minimal torus $T \subset \Aut(X)$ such that $\Lie T \supset S$, 
and that $T$ normalizes $R_n$. By Lemmas~\ref{toral.lem} and \ref{J-saturation.lem}, $T$ normalizes $U$ and $\tilde R_n$,  
and so the product $Q:=T U = UT \subset \Aut(X)$ is a connected solvable subgroup with $\Lie Q \supset \tilde R$. 
By construction, $Q$ is the smallest subgroup with this property. 
It remains to show that $G_0$ normalizes $Q$. Indeed,  $gQg^{-1}$ is a closed subgroup for any $g \in G_0$,  
and $\Lie gQg^{-1} \supset g\tilde Rg^{-1} = \tilde R$.

The uniqueness statement follows from the fact that $\Lie (G \cap G') = \Lie G \cap \Lie G'$ for any algebraic subgroups 
$G, G' \subset \Aut(X)$, see \cite[Proposition~7.7.1]{FuKr2018On-the-geometry-of}.

\end{proof}
\ps
\section{Nested ind-groups and nested subgroups}
\subsection{Nested groups and nested Lie algebras} 
For the following concept see \cite[Section 9.4]{FuKr2018On-the-geometry-of} and \cite{KoPeZa2017On-automorphism-gr}.

\begin{defn}\label{nested.def1}
An ind-group $\G$ is called {\it nested} if it admits an admissible filtration consisting of algebraic subgroups.
A Lie algebra $L$ is called {\it nested\/} if there exists a sequence 
$L_1 \subset L_2 \subset \cdots \subset L$ of finite dimensional Lie algebras $L_i$ such that $L  = \bigcup_i L_i$.\end{defn}
\begin{lem}\label{nested.lem1}
\be
\item
The Lie algebra $\Lie\G$ of a nested ind-group $\G$ is nested.
\item 
Every filtration of a nested ind-group by algebraic groups is admissible.
\item\label{item3}
A connected nested ind-group admits a filtration by connected algebraic subgroups. In particular, it is algebraically generated (Definition~\ref{AGgroup.def}).
\item
Assume that the base field $\kk$ is uncountable. If an ind-group $\G$ has a filtration by algebraic subgroups, then $\G$ is nested.
\ee
\end{lem}
\begin{proof}
(1) This is clear.
\ps
(2) Let $\G = \bigcup_k \G_k$ be a nested ind-group, i.e. the $\G_k$ are algebraic subgroups, and let $G_1 \subset G_2 \subset \cdots \subset \G$ be a sequence of algebraic subgroups such that $\G = \bigcup_i G_i$. Then each $G_i$ is contained in some $\G_k$, and $\G_k = \bigcup_i (\G_k \cap G_i)$ is a countable union of closed subgroups. It follows from Lemma~\ref{finite-genertation.lem} below that $\G_k = \G_k \cap G_i$ for some $i$ and so $\G_k \subset G_i$.
\ps
(3) If $\G = \bigcup_k \G_k$ is connected, then $\G = \bigcup_k \G_k^\circ$ by \cite[Proposition~2.2.1]{FuKr2018On-the-geometry-of}.
\ps
(4) If $\kk$ is uncountable, then every filtration by closed algebraic subsets is admissible, see \cite[Theorem~1.3.3]{FuKr2018On-the-geometry-of}.
\end{proof}
The following lemma is well-known.
\begin{lem}\label{finite-genertation.lem}
For any algebraic group $G$ there exists a finite set $\{g_1,g_2,\ldots,g_n\}\subset G$ such that $G = \overline{\langle g_1,g_2,\ldots,g_n\rangle}$. 
\end{lem}
\begin{proof}[Outline of Proof]
The statement is clear for a unipotent group $U$ since $\overline{\langle u \rangle} \simeq \Ga$ for a unipotent element $u\neq e$. It is also clear for a torus $T$ since there is always an element $t \in T$ such that $T = \overline{\langle t \rangle}$. Finally, for a semisimple group $G$ we fix a maximal torus $T$, choose a ``dense'' element $t \in T$ and unipotent elements in every root subgroup of $G$.
\end{proof}
\begin{defn}\label{nested.def2}
A subgroup $G$ of an ind-group $\G$ is called {\it quasi-nested\/} if there exists a sequence  $G_1 \subset G_2 \subset \cdots \subset \G$ of bounded subgroups of $\G$ (see section~\ref{VF.subsec}) such that $G = \bigcup_i G_i$. It is called {\it nested} if the $G_i$ are algebraic subgroups.
\end{defn}
The situation here is subtle. Every nested subgroup $G$ of an ind-group $\G$ has a natural structure of a (nested) ind-group given by the filtration by algebraic subgroups, and thus $\Lie G$ is well defined and is a subalgebra of $\Lie \G$. But in general $G$ is not a closed ind-subgroup as shown by the following example.
\begin{exa}
The subgroup $\mu_2 \subset \Gm$ of elements of order a power of $2$ is nested by the finite subgroups $\simeq \ZZ/2^n\ZZ$. But it is not closed.
\end{exa}

\begin{rems}\label{nested.rems}
\be
\item[(a)]
A quasi-nested subgroup $G = \bigcup_i G_i \subset \G$ defines a nested subgroup, namely $\bigcup_i \overline{G_i}$.
\item[(b)]
Every subgroup of a nested ind-group is quasi-nested.
\item[(c)]
Let $\phi\colon \G \to \H$ be a homomorphism of ind-groups. If $\G$ is nested, then the image $\phi(\G) \subset \H$ is a nested subgroup.
\item[(d)]
If the subgroup $G \subset \G$ is quasi-nested, $G = \bigcup_i G_i$, then one can always assume that the subgroups $G_i \subset G$ are closed in $G$.
\ee
\end{rems}
\begin{prop}
Assume that the base field $\kk$ is uncountable.
Then a path-connected quasi-nested subgroup $G \subset \G$ is nested.
\end{prop}
\begin{proof}
Let $G = \bigcup G_i \subset \G$ where the $G_i$ are closed  in $G$, see Remark~\ref{nested.rems}(d). Define $G_i^\circ \subset G_i$ to be the path-connected component of $e$, i.e. the set of elements of $G_i$ which can be connected to $e \in G_i$. Then $G_i^\circ$ is a path-connected bounded subgroup, hence an algebraic group by Corollary~\ref{path-connected-subgroup.cor}. 

We claim that $G = \bigcup_i G_i^\circ$. For any $g \in G$ there is an irreducible variety $Y$ and a morphism $\phi\colon Y \to \G$ such that $e,g\in \phi(Y) \subset G$. Since the subgroups $G_i \subset G$ are closed, the inverse images $\phi^{-1}(G_i) \subset Y$ are closed, and $\bigcup_i \phi^{-1}(G_i) = Y$. Since $\kk$ is uncountable there is an $i$ such that $Y = \phi^{-1}(G_i)$. Hence $\phi(Y) \subset G_i^\circ$ and so $ g \in\bigcup_i G_i^\circ$.
\end{proof}

\begin{que}\label{Q.4}
Let $G \subset \G$ be a path-connected nested subgroup, e.g. a unipotent nested subgroup. 
Is it true that $G$ is closed? {\rm(Cf. \cite{Pe2023Structure-of-conne}.)}
\end{que}
In order to give a positive answer to this question it suffices to show that for every closed algebraic subset $A \subset \G$ the intersection $A\cap G$ is contained in an algebraic subgroup of $G$. In fact, if $A\cap G \subset G_i$, then $A \cap G = A \cap G_i$ is closed in $G_i$.
\ps
The next result is an immediate consequence of Proposition~\ref{orbit-agg.prop} and Lemma~\ref{nested.lem1}(\ref{item3}).
\begin{cor}\label{orbits-nested.cor}
Let $G = \bigcup_k G_k\subset \Aut(X)$ be a path-connected nested subgroup. Then there is a $k_0$ such that for $k\geq k_0$ the groups $G$ and $G_k$ have the same orbits on $X$.
\end{cor}
\ps
\subsection{Integration of nested Lie algebras}
If $G \subset \Aut(X)$ is a nested subgroup, then $\Lie G \subset \VF(X)$ is a nested Lie algebra filtered by locally finite subalgebras. 
The converse of this is also true.
\begin{prop}
Let $L = \bigcup_i L_i \subset \VF(X)$ be a nested Lie subalgebra. 
Then the following hold.
\be
\item  If $L$ is generated by locally nilpotent vector fields, then
there exists a nested unipotent subgroup 
$G=\bigcup_i G_i\subset\Aut(X)$ such that $L=\bigcup_i\Lie G_i$.
\item  Assume that all $L_i$ are locally finite. 
Then there exists a smallest nested subgroup $G \subset \Aut(X)$ such that $L \subset \Lie G \subset \VF(X)$. 
Moreover, a closed ind-subgroup $\G \subset \Aut(X)$ contains $G$ if and only if  $L \subset \Lie\G$.
\ee
\end{prop}
\begin{proof} 
(1) Any finite collection $\SSS\subset L$ of locally nilpotent vector fields 
generates a Lie subalgebra, say, $M_\SSS$ of some $L_j$. According to Theorem~\ref{CD.thm}(1), 
$M_\SSS=\Lie G_\SSS$ for an algebraic subgroup
$G_\SSS$ of $\Aut(X)$. On the other hand, $L_i$ is contained in some $M_i:=M_{\SSS_i}$, hence $L = \bigcup_i M_i$. 
We may suppose that $M_i\subset M_{i+1}$, and so $G_i\subset G_{i+1}$ where $G_i:=G_{\SSS_i}$. The union $G:=\bigcup_i G_i$ is a nested subgroup of $\Aut X$ such that 
$L= \bigcup_i \Lie G_i$.
\ps
(2) Let $G_i \subset \Aut(X)$ be the smallest algebraic subgroup with $\Lie G_i \supset L_i$ (see Theorem~\ref{CD.thm}(2)). 
Then $G_i \subset G_{i+1}$ for all $i$ and so $G:= \bigcup_i G_i$ is a nested subgroup with the required properties.

Now assume that  $\Lie\G \supset L$ and define $G_i' := \G \cap G_i$. Then $\Lie G_i' = \Lie\G \cap \Lie G_i \supset L_i$,  hence $G_i' = G_i$ and the claim follows.
\end{proof}
\ps
\section{Solvability and triangulation}
\subsection{Solvable subgroups}\label{solvable.subsec}
This section is devoted to the proof of our Theorem~B from the introduction.
\begin{thm}\label{thmB-new}
Let $G = \langle H_i\mid i \in I\rangle$ be a solvable algebraically generated group. If the family $I$ is finite, then $G$ is a solvable algebraic group 
with Lie algebra $\Lie G = L(G) =\langle \Lie H_i\mid i \in I\rangle$. In general, $G$ is nested and is of the form $G = U_G \rtimes T$ 
where $T \subset \Aut(X)$ is a torus and $U_G$ is a nested unipotent group consisting of the unipotent elements of $G$.
\end{thm}
\begin{cor}\label{thmB.cor}
A solvable subgroup $G \subset \Aut(X)$ generated by unipotent elements is quasi-nested and consists of unipotent elements.
\end{cor}
\begin{proof}[Proof of Corollary]
If $G$ is generated by the unipotent elements $(u_i)_{i\in I}$, then the closure $\overline{G}\subset \Aut(X)$ contains 
the subgroup $\widetilde G$ generated by the unipotent subgroups $\overline{\langle u_i \rangle}$. 
We will see in the following Lemma~\ref{solvableGr.lem}(2) that $\overline{G}$ is also solvable. 
It then follows from the theorem above that $\widetilde G=U_{\widetilde G} \rtimes T$ is nested, hence $G$ is quasi-nested and is contained in $U_{\widetilde G}$. 
\end{proof}
The proof of Theorem \ref{thmB-new} needs some preparation. 
Given a solvable group $G$ we denote by $G^{(i)}$ the members of the derived series of $G$ and by $d(G)$ the derived length:
$$
G^{(0)} := G \supset G^{(1)} := (G,G) \supset G^{(2)} := (G^{(1)},G^{(1)}) 
\supset \cdots \supset G^{(d)} = \{e\}.
$$
If $G$ is nilpotent, then $G_i$ stands for the members of the lower central series of $G$ and $n(G)$ for its  nilpotency class:
$$
G_0:=G \supset G_1:=(G,G)\supset G_2 := (G,G_1) \supset \cdots \supset G_n=\{e\}.
$$
In the following lemma we collect some important facts about closures of solvable and nilpotent groups. 
Statement (2) can be found in \cite[ Lemma 2.3(3)]{FuPo2018On-the-maximality-}. 
\begin{lem}\label{solvableGr.lem}
Let $G, H \subset \G$ be subgroups of an ind-group $\G$.
\be
\item We have $(\overline{G},\overline{H}) \subset \overline{(G,H)}$, 
$\overline{G}^{(i)}\subset \overline{G^{(i)}}$ and $\overline{G}_i\subset \overline{G_i}$.
\item If $G$ is solvable, then so is $\overline{G}$, and $d(\overline{G})=d(G)$.
\item If $G$ is nilpotent, then so is $\overline{G}$, and $n(\overline{G})=n(G)$.
\ee
\end{lem} 
\begin{proof}
(1) 
Let $\gamma\colon \G \times \G \to \G$ be the commutator map, $(g,h) \mapsto ghg^{-1}h^{-1}$.
Then  $\gamma(\overline{G}\times\overline{H}) \subset \overline{\gamma(G \times H)}$, and so
$(\overline{G},\overline{H}) \subset \langle \overline{\gamma(G \times H)}\rangle \subset \overline{(G,H)}$.
Now the remaining inclusions follow.
\ps
(2)
By (1) we have $G^{(i)} \subset \overline{G}^{(i)} \subset \overline{G^{(i)}}$. Hence $\overline{G}^{(d(G))} = \{e\}$ and 
$\overline{G}^{(i)} \neq \{e\}$ for $i<d(G)$.
\ps
(3) 
By (1) we have $G_i \subset \overline{G}_i \subset \overline{G_i}$. Hence $\overline{G}_{n(G)} = \{e\}$ and 
$\overline{G}_i \neq \{e\}$ for $i<n(G)$.
\end{proof}
\ps
\begin{lem}\label{solvableLA.lem} 
Let $\G$ be an ind-group and $\H$ be a  closed subgroup of $\G$. 
Then the following hold.
\be
\item  $[\Lie \H,\Lie \G]\subset \Lie\overline{(\H,\G)}$.
\item If $\G$ is solvable, then so is $\Lie \G$.
\item If $\G$ is nilpotent, then so is $\Lie \G$. 
\ee
\end{lem}
\begin{proof} (Cf. \cite[Section~7.5]{FuKr2018On-the-geometry-of})
\ps
(1)  For a fixed $h \in \H$ consider the morphism 
\[\gamma_h\colon \G \to (\H,\G),\quad \gamma_h(g):= hgh^{-1}g^{-1}.\] 
The differential $(d\gamma_h)_e \colon \Lie \G \to \Lie \overline{(\H,\G)}$  is given by $(d\gamma_h)_e = \Ad(h) - \id$. 
Now fix $A \in \Lie\G$ and consider the morphism 
\[\alpha_A\colon \H \to \Lie \G,\quad \alpha_A(h) := \Ad(h)A - A=d\gamma_h (A).\] 
 The differential at $e$ of this map is $ -\ad(A)$; it sends $\Lie \H$ onto $[\Lie \H, A]$.
By the preceding, $(d\gamma_h)_e (A)\in\Lie \overline{(\H,\G)}$ for any $h\in\H$ and $A \in \Lie\G$. 
Hence, $[\Lie \H, A]\subset\Lie \overline{(\H,\G)}$ for any $A \in \Lie\G$.
This yields (1). 
\ps
(2) Now assume that $\G$ is solvable. From the derived series for $\G$ we get the normal series
$$
\G^{(0)} = \G \supset \overline{\G^{(1)}}\supset \overline{\G^{(2)}} \supset \cdots \supset \overline{\G^{(n)}} = \{e\}
$$
of closed subgroups with the property that the factor groups are all commutative (see Lemma~\ref{solvableGr.lem}(1)). Passing to the series
\[
\Lie \G \supset \Lie\overline{\G^{(1)}}\supset \Lie\overline{\G^{(2)}} \supset \cdots \supset \Lie\overline{\G^{(n)}} = 0
\]
we get from (1)
$$
\Lie\overline{\G^{(i+1)}} = \Lie\overline{(\G^{(i)},\G^{(i)})} = \Lie\overline{(\overline{\G^{(i)}},\overline{\G^{(i)}})}
\supset [\Lie\overline{\G^{(i)}},\Lie\overline{\G^{(i)}}].
$$
Thus the factors $\Lie\overline{\G^{(i)}}/\Lie\overline{\G^{(i+1)}}$ are all commutative, which implies that $\Lie\G$ is solvable.
\ps
(3) Assume now that $\G$ is nilpotent. From the lower central series for $\G$ we get the series of closed subgroups 
$$
\G_0 = \G \supset \overline{\G_1}\supset \overline{\G_2} \supset \cdots \supset \overline{\G_n} = \{e\}
$$ 
and the corresponding series
\[\tag{$*$}
\Lie \G \supset \Lie\overline{\G_1}\supset \Lie\overline{\G_2} \supset \cdots \supset \Lie\overline{\G_n} = \{0\}.
\]
By virtue of (1) we have
$$
\Lie\overline{\G_{i+1}} = \Lie\overline{(\G,\G_{i})} = \Lie\overline{(\overline{\G},\overline{\G_{i}})}
\supset [\Lie\overline{\G},\Lie\overline{\G_{i}}].
$$
Thus $(*)$ is a central series and so $\Lie\G$ is nilpotent.
\end{proof}
\begin{que} 
Let $\G$ be a path connected ind-group. 
Is it true that $\G$ is solvable (resp. nilpotent) if $\Lie\G$ is? 
\end{que}

\begin{rem}\label{solvable.rem} 
It is known that a path-connected ind-group $\G$ 
is commutative if $\Lie\G$ is, see \cite[Corollary 7.5.3]{FuKr2018On-the-geometry-of}. 
\end{rem}

\begin{lem}\label{finitelygeneratedLA.lem}
Let $L$ be a finitely generated Lie algebra. If $L$ is solvable, then $L$ is finite dimensional.
\end{lem}
\begin{proof}
In the derived series
$$
L^{(0)}:=L \supset L^{(1)}:=[L^{(0)},L^{(0)}] \supset L^{(2)}:=[L^{(1)},L^{(1)}] \supset \cdots \supset L^{(n)} = \{0\}
$$
every member $L^{(i)}$ is finitely generated. Therefore $L^{(i)}/L^{(i+1)}$ is a commutative and finitely generated Lie algebra, 
hence finite dimensional, and the claim follows.
\end{proof}

\begin{proof}[Proof of Theorem~\ref{thmB-new}] 
(a) We first consider the case where $I$ is finite.
The closure $\overline{G}$ is solvable (Lemma~\ref{solvableGr.lem}(2)) and so $\Lie\overline{G}$ is solvable 
(Lemma~\ref{solvableLA.lem}(2)). It follows that $L(G)= \langle \Lie H_i\mid i=1,\ldots,m\rangle_{\Lie}$ is also solvable 
(Theorem~\ref{main1.thm}(\ref{item1})), hence finite dimensional (Lemma~\ref{finitelygeneratedLA.lem}). Now the claim follows from Theorem~\ref{main2.thm}.
\ps
(b) 
In general, it follows from (a) that for any finite subset $F \subset I$ the subgroup $G_F:=\langle H_i \mid i \in F\rangle$ is a connected solvable algebraic subgroup of $\G$. If $I$ is countable we have an ascending filtration $F_0 \subset F_1 \subset F_2 \subset \cdots \subset I$ by finite subsets, and so $G = \bigcup_j G_{F_j}$ is nested.
\ps
(c)
For a general $I$ we first remark that the Lie algebra $L \subset \VF(X)$ generated by the $\Lie H_i$ has countable dimension, and so $L = \langle \Lie H_j\mid j\in J\rangle_{\Lie}$ for a countable subset $J \subset I$. We claim that $G$ is generated by the $H_j$, $j \in J$, which implies by (b) that $G$ is nested. In fact, 
for any $H_i$ there are finitely many $H_{j_1},\ldots,H_{j_m}$ with $j_k \in J$ such that $\Lie H_i \subset \langle \Lie H_{j_k}\mid k=1,\ldots,m\rangle_{\Lie}$. Thus $\Lie H_i \subset \Lie \tilde{H}_i$ where $\tilde{H}_i$ is the solvable algebraic subgroup generated by the $H_{j_k}$, and so  $H_i \subset \tilde{H}_i$ by \cite[Remark~17.3.3]{FuKr2018On-the-geometry-of}.
\ps
(d) 
If the nested subgroup $G=\bigcup_k G_k$ does not contain semisimple elements, then $G=U_G$ is a nested unipotent group. Otherwise we choose a torus $T \subset G$ of maximal dimension. We can assume that $T$ is contained in all $G_k$, hence $G_k = U_{G_k}\rtimes T$ for all $k$, and $U_{G_k} \subset U_{G_{k+1}}$. Thus $U_G := \bigcup_k U_{G_k}$ is a nested unipotent group containing all unipotent elements of $G$, and $G = U_G\rtimes T$.
\end{proof}
\ps
\subsection{Triangulation in \texorpdfstring{$\Aut(\An)$}{Aut(An)}}
By definition, an automorphism $\phi$ of $\An$ belongs to the  de Jonqui\`eres subgroup $\Jonq(n)$ (see Section~\ref{solvability-triangulation.subsec}) if and only if the comorphism $\phi^*$ stabilizes the flag
$$
\kk[x_1,\ldots,x_{n-1},x_n] \supset \kk[x_1,\ldots,x_{n-1}] \supset\cdots\supset \kk[x_1,x_2] \supset \kk[x_1].
$$
Equivalently, $\phi$ stabilizes the coflag
$$
\FFF \colon \An \to \AA^{n-1} \to \AA^{n-2} \to \cdots \to \AA^2 \to \AA^1
$$
which means that $\phi$ induce an automorphism on each $\AA^k$ such that the projections $\AA^k \to \AA^{k-1}$, $(a_1,\ldots,a_n)\mapsto (a_1,\ldots,a_{n-1})$ are $\phi$-equivariant.
\begin{rem}\label{rem1}
$\Jonq(n)$ is a connected, nested solvable group of the form $\Jonq(n) = \Jonq(n)_u \rtimes T$ where $T := \kst^n$ is the standard maximal torus and $\Jonq(n)_u$ is the subgroup of unipotent elements, cf. \cite[Section~15.1]{FuKr2018On-the-geometry-of}. $\Jonq(n)$ has derived length $d(\Jonq(n)) = n+1$  (\cite[Lemma~3.2]{FuPo2018On-the-maximality-}) and thus $d(\Jonq(n)_u) = n$, because $(\Jonq(n),\Jonq(n)) = \Jonq(n)_u$.
\end{rem}
\begin{thm}\label{thmD}
Let $U \subset \Aut(\An)$ be a nested unipotent subgroup. If $U$ has a dense orbit on $\An$, then $U$ acts transitively on $\An$, and $U$ is conjugate to a subgroup of $\Jonq(n)_u$. 
\end{thm}
\begin{proof}
The first statement follows Corollary~\ref{orbits-nested.cor} and the fact, that $U$-orbits are closed.
\ps
For the second claim
we have to construct a $U$-stable coflag $\An \to \AA^{n-1} \to \cdots\to \Aone$ where the projections are of the form $\AA^k \simeq \AA^{k-1} \times \Aone \overset{\text{pr}}{\longrightarrow} \AA^{k-1}$.

Let $U_x \subset U$ be the isotropy group of some point $x \in \An$. By Lemma~\ref{normaliser.lem} below we can find an element $a \in \Norm_U U_x \setminus U_x$. Then the subgroup $A:=\overline{\langle a \rangle}$ is  isomorphic to $\Ga$ and normalizes $U_x$. Then there is a right action of $A$ on $U/U_x \simto \An$ which commutes with the $U$-action. Thus we get an action of $U$ on the quotient $\An\quot A:=\Spec\OOO(\An)^A$ such that the quotient map $p\colon \An \to \An\quot A$ is $U$-equivariant. It follows that $U$ has a dense orbit in $\An\quot A$ and so $\An\quot A \simeq \AA^{n-1}$, because orbits of unipotent groups are closed and isomorphic to affine spaces, cf. \cite[Theorem~11.1.1]{FuKr2018On-the-geometry-of}.

Moreover, the $A$-action on $\An$ admits a local slice, hence a local section of $p$, see \cite[Section 1.4, Principle 1(c)]{Fr2006Algebraic-theory-o}. Since $U$ acts transitively on $\An$ and $p$ is $U$-equivariant there exist  local sections of $p$ in any $z\in \An$ and thus $p\colon \An \to \An\quot A$ is a principal $A$-bundle. By a theorem of Serre \cite[Section 5.1]{Se1958Espaces-fibres-alg} we have $H^1(X,\Ga)=0$ for any affine variety $X$, and hence this bundle is trivial:
$$
\begin{CD}
\An @>{\simeq}>> \AA^{n-1} \times \AA^1 \\
@VV{p}V @VV{\pr}V \\
\An\quot A  @>{\simeq}>>  \AA^{n-1}
\end{CD}
$$
Thus, we have constructed a $U$-equivariant projection $\An \to \AA^{n-1}$, and the claim follows by induction.
\end{proof}
\begin{lem}\label{normaliser.lem}
Let $U \subset \Aut(X)$ be a nested unipotent subgroup. For any point $x \in X\setminus X^U$ the normalizer $\Norm_U U_x$ strictly contains the isotropy group $U_x$.
\end{lem}
\begin{proof}
The claim of the lemma is well-known for a unipotent algebraic group $U$ since such a group is nilpotent. 
In general, $U = \bigcup_k U_k$ with closed unipotent algebraic subgroups $U_k$ such that $U_k \subset U_{k+1}$, and, similarly, $U_x = \bigcup_k (U_k)_x$ with closed unipotent algebraic subgroups $(U_k)_x = U_x \cap U_k$. Corollary~\ref{orbits-nested.cor} implies that we can find a $k_0$ such that 
$(U_k)_x$ has the same orbits on $X$ as $U_x$ for all $k\geq k_0$.

Now we use the following general fact. If a group $G$ acts on a space $X$, and if $x \in X$, then $g \in G$ belongs to the normalizer of the isotropy group $G_x$ if and only if the isotropy group of $gx$ is equal to $G_x$, i.e. if and only if $gx \in X^{G_x}$.

Assume that $g \in U_k$ belongs to the normalizer of $(U_k)_x$. Then $gx \in X^{(U_k)_x} = X^{U_x}$, and so $g$ is in the normalizer of $U_x$. If $g \notin (U_k)_x = U_k \cap U_x$, then $g \notin U_x$, and the claim follows. 
\end{proof}

It is a basic fact in algebraic transformation groups that every affine $G$-variety $X$ admits a closed $G$-equivariant embedding $X \into V$ where $V$ is $G$-module. The corresponding statement for ind-groups does not hold, see \cite[Proposition~2.6.5]{FuKr2018On-the-geometry-of}.

\begin{que}\label{triangulation}  Let $G \subset \Aut(X)$ be a solvable or nilpotent connected subgroup. Does there exists a closed embedding $X \into \An$ such that $G$ extends to a subgroup of the de Jonqu\`eres group $\Jonq(n)$? Is this true if $G$ is nested?
\end{que}
\ps
\subsection{Nested unipotent subgroups of \tp{$\Aut(X)$}{Aut(X)} are solvable}\label{nested-unipotent.subsec}
We know from Theorem B that a solvable subgroup 
$G$ of $\Aut(X)$ generated by a family of unipotent 
algebraic groups is a nested unipotent group.
The next result shows that the converse holds as well.
\begin{thm}\label{thmC} 
A nested unipotent subgroup $U \subset \Aut(X)$ is 
solvable of derived length $\leq \max\{\dim Ux\mid x \in X\}\leq \dim X$.
\end{thm}
\begin{proof} 
It suffices to consider the case of a unipotent algebraic subgroup $U \subset \Aut(X)$.  Then every orbit $O=Ux$ is closed and isomorphic to an affine space. Therefore, by Theorem \ref{thmD} and Remark~\ref{rem1}, the image of $U$ in $\Aut(O)$ has derived length $\leq \dim O$. If $m := \max\{\dim Ux \mid x \in X\}$, then the $m$th member $U^{(m)}$ of the derived series of $U$ acts trivially on every orbit, hence on $X$, and so $U^{(m)} = \{e\}$.
\end{proof}

\begin{cor}\label{cor:nested-solvable}
\be
\item
A solvable algebraically generated subgroup $G\subset\Aut(X)$
is of derived length $\le \dim X+1$.
\item
Let $\G=\bigcup_k \G_k\subset\Aut(X)$ be a connected nested subgroup where
the $\G_k$ are solvable algebraic groups. Then $\G$ is solvable 
of derived length $\le \dim X+1$.
\ee
\end{cor}
\begin{proof}
(1) By Theorem~\ref{thmB-new}, $G=U_G\rtimes T$ 
where $T\subset G$ is a maximal torus 
and $U_G$ is a nested unipotent subgroup. 
By Theorem~\ref{thmC}, $G$ is solvable of derived length $\le \dim X+1$.
\ps
(2) By Lemma~\ref{nested.lem1}(3) we may suppose that all the $\G_k$ are connected. 
Thus, $\G_k=U_k\rtimes T_k$ 
where $U_k=R_u(\G_k)$ and $T_k$ is the maximal torus of $\G_k$. 
Clearly, we have $U_k\subset U_{k+1}$, and we can arrange that 
$T_k\subset T_{k+1}$ for all $k$. 
Then again $\G=U\rtimes T$ where $U:=\bigcup_k U_k$ and $T:=\bigcup_k T_k$ 
is a maximal torus of $\G$ of dimension $\le \dim(X)$.
Now the assertion follows from Theorem~\ref{thmC} above.
\end{proof}

\begin{rem}\label{rem2} 
A unipotent algebraic group is nilpotent, but the nilpotency class of unipotent subgroups of $\Aut(X)$ might not be bounded, and thus a nested unipotent subgroup is not necessarily nilpotent.

For example, $\JJJ_n:=\Jonq(n)_u$ is a closed nested unipotent subgroup of $\Aut(\An)$. For $n\ge 2$ its Lie algebra $\Lie\JJJ_n$ is not nilpotent, and so $\JJJ_n$ is neither, by Lemma \ref{solvableLA.lem}(3). 
Consider, for instance, the case $n=2$, and define 
\[
L_d := \langle \partial/\partial x_1, x_1^d\partial/\partial x_2\rangle\subset \Lie\JJJ_2.
\]
It is easily seen that the $d$th member $(L_d)_d$ of the lower central series of $L_d$ does not vanish, and so $(\Lie\JJJ_2)_d\neq 0$ for any $d\geq 1$. 
\end{rem}

\ps
\section{An interesting example}\label{sec:example}
In this section we study the subgroup $F \subset \Aut(\Atwo)$ generated by $\bu:=(x+y^2,y)$ and $\bv:=(x,y+x^2)$. 
We will show that $F$ is a free group in two generators and we will describe the closure $\F:=\overline{F}$ and its Lie algebra $\Lie \F$.
It turns out that the following holds (see Theorem~\ref{example.thm}).
\be
\bullitem
$\F$ is a free product of two nested closed abelian unipotent ind-subgroups $\J$ and  $\J^- \subset \Aut(\Atwo)$. In particular, $\F$ is torsion free.
\bullitem
As a Lie algebra, $\Lie \F$ is generated by $\Lie U$ and $\Lie V$ where $U:=\overline{\langle \bu \rangle} \simeq \Ga$ 
and $V:=\overline{\langle \bv \rangle} \simeq \Ga$.
\bullitem 
Any algebraic subgroup of $\F$ is abelian and unipotent and is conjugate to a  subgroup of $\J$ or $\J^-$.
\ee
\ps
\subsection{Notation and first results}\label{notation}
Let 
\[
T:=\{(ax,by) \mid a,b \in \kk^*\} \subset \GL_2(\kk)\subset\Aut(\Atwo)
\] 
denote the standard 2-dimensional torus, and let
\[
\SAut(\Atwo):= \{\bg \in \Aut(X) \mid \Jac(\bg) = 1\}
\]
where the {\it Jacobian} of $\bg = (f,h)$ is defined in the usual way: 
$\Jac(\bg) = \det \left[\begin{smallmatrix}  \frac{\partial f}{\partial x} & \frac{\partial f}{\partial y} \\ \frac{\partial h}{\partial x} & \frac{\partial h}{\partial y} \end{smallmatrix}\right]$. 
Furthermore, we set

$$
\Aut_0(\Atwo) := \{\bg \in \Aut(\Atwo) \mid \bg(0) = 0\},
$$
the group of automorphisms fixing the origin, and 
$$
\SAut_0(\Atwo):=\SAut(\Atwo)\cap\Aut_0(\Atwo).
$$
There is a canonical homomorphism $D\colon\Aut_0(\Atwo) \to \GL_2(\kk)$ defined by $D(\bg):=(d\bg)_0$.
If $\bg=(f,h)$, then $\bg \in \Aut_0(\Atwo)$ if and only if $f(0)=h(0)=0$, i.e. $f$ and $h$ have no constant terms, and $D(\bg) = (f_1,h_1)$, the endomorphism given by the linear terms of $\bg\in\Aut_0(\Atwo)$.
We will be interested in the kernel of $D$:
\begin{align*}
\A:=\ker D &= \{\bg \in \Aut_0(\Atwo) \mid (d\bg)_0 = \id\}\\
&= \{(f,h) \in \Aut_0(\Atwo) \mid (f_1,h_1) = (x,y)\}.
\end{align*}
We get a split exact sequence of ind-groups
\[\tag{$*$}
\begin{CD}
1 @>>> \A @>>> \Aut_0(\Atwo) @>D>> \GL_2(\kk) @>>> 1
\end{CD}
\]
Thus $\A$ is a closed path-connected ind-subgroup of $\Aut(\Atwo)$ which contains
$\bu=(x+y^2,y)$ and $\bv=(x,y+x^2)$, and so $\F \subset \A$. 

Let $S:=\{(\zeta x, \zeta^2 y)\mid \zeta^3=1\} \subset T\cap \SLtwo \subset \SAut(\Atwo)$ be the cyclic diagonal subgroup of order 3.
The elements of $S$ commute with $\bu$ and $\bv$ and thus with all elements from $F$ and hence from $\F=\overline{F}$. 

\begin{lem}\label{example.lem} 
One has  $S=\Cent_{\Aut(\Atwo)}(F)$.
\end{lem}

\begin{proof}
Note that any $\bg \in \Aut(\Atwo)$ commuting with $\bu$ and $\bv$, and therefore with $U$ and $V$, 
fixes the origin which is the unique common fixed point of $U$ and $V$. It also leaves invariant the set of orbits of $U$ and of $V$, 
i.e. the pencils of horizontal and vertical lines. Hence, $\bg \in T$, and an easy calculation shows that, indeed, $\bg\in S$.
\end{proof}

Considering the action of $S$ on $\Aut(\Atwo)$ by conjugation, one sees that $\SAut(\Atwo)$ and $\A$ are stable under this action, and thus
\begin{equation*}
\F \subset \A^S := \{\bg \in \A \mid \bg \circ s = s \circ \bg \text{ for all } s \in S\}.
\end{equation*}
Moreover, $S$ also acts on the Lie algebras $\Lie \Aut(\Atwo)$ and $\Lie\A$, and so
\beq\label{G-equ}
\Lie \F \subset \Lie \A^S \subset (\Lie\A)^S.
\eeq
\ps
\subsection{The Lie algebra of \tp{$\F$}{F}}\label{Lie-algebra.subsec}
The {\it divergence $\Div$} of a vector field $\delta = f \frac{\partial}{\partial x} + h \frac{\partial}{\partial y}$ 
is defined as $\Div \delta = \frac{\partial f}{\partial x} + \frac{\partial h}{\partial y}$. 
For the Lie algebras of $\Aut(\Atwo)$, $\SAut(\Atwo)$ and $\Aut_0(\Atwo)$ we have the following description:
\be
\bullitem
$\Lie\Aut(\Atwo) \simto \VF^c(\Atwo):=\{\delta\in\VF(\Atwo) \mid \Div \delta \in \kk\}$,
\bullitem 
$\Lie\SAut(\Atwo) \simto \VF^0(\Atwo):=\{\delta\in\VF(\Atwo) \mid \Div \delta =0 \}$,
\bullitem
$\Lie\Aut_0(\Atwo) \simto \VF_0(\Atwo):= \{\delta\in\VF^c(\Atwo) \mid \delta(0) = 0\}$,
\ee
The first two are given in \cite[Proposition~15.7.2]{FuKr2018On-the-geometry-of}, and the last is an immediate consequence. The split exact sequence $(*)$ from the previous section~\ref{notation} implies that 
\begin{align*}
\Lie \A &= \ker (dD_e\colon \VF_0(\Atwo) \to \gl_2)\\
&= \{\delta=f {\frac{\partial}{\partial x}} + h {\frac{\partial}{\partial y}}\in\VF_0(\Atwo) \mid f_1=h_1=0\}.
\end{align*}
Now consider the following vector fields from $\VF^0(\Atwo)$,
\[
\p_{i,j}: =(j+1)x^{i+1}y^{j}\frac{\p}{\p x} - (i+1)x^{i}y^{j+1}\frac{\p}{\p y}
\]
where
\[
(i,j) \in \Lambda := \{(k,l)\in \ZZ^2\,\mid k,l\geq -1, k+l \geq 0 \}.
\]
The torus $T$ acts by conjugation on $\Aut(\Atwo)$ and on $\VF(\Atwo)$, and the $\p_{i,j}$ are 
eigenvectors of weight $(i,j)$. Moreover, the action of $T$ on $\VF^0(\Atwo)$ is multiplicity-free, and we have the weight decompositions
\[
\VF^0(\Atwo) = \bigoplus_{(i,j)\in \Lambda} \kk \p_{i,j} \text{ \ and \ }
\Lie\A = \bigoplus_{(i,j)\in \Lambda, i+j>0} \kk \p_{i,j}.
\]
Note that $\p_{i,j}$ is fixed by the finite subgroup $S\subset T$ if and only if $i-j\equiv 0 \mod 3$. 
Setting $\Lambda_0:=\{(k,l)\in\Lambda \mid k-l \equiv 0 \mod 3, (k,l)\neq (0,0)\}$ we get
\beq\label{Lie-AS.eq}
\Lie \A^S \subset (\Lie \A)^S = \bigoplus_{(i,j)\in \Lambda_0} \kk \p_{i,j}.
\eeq
The next lemma shows that $\Lie\F = \Lie \A^S = (\Lie\A)^S$. In contrast to the situation of algebraic groups, 
this does not automatically  imply that $\F = \A^S$. In fact, \name{Furter-Kraft} give an example of a closed ind-subgroup 
of a path-connected ind-group which has the same Lie algebra, but is strictly smaller (see \cite[Theorem~17.3.1]{FuKr2018On-the-geometry-of}). 
In order to show that $\F = \A^S$ (see Theorem~\ref{example.thm}) we therefore need an additional argument.
\begin{lem}\label{example-lie.lem}
The Lie algebra $(\Lie\A)^S$ is generated by $\Lie U$ and $\Lie V$, hence $\Lie \F = \Lie \A^S = (\Lie\A)^S$.
\end{lem}
\begin{proof}
We have $\Lie U=\kk\p_{-1,2}$ and $\Lie V=\kk\p_{2,-1}$. An easy calculation shows that 
\[
\ad\p_{-1,2} (\p_{i,j})=a\,\p_{i-1,j+2} \, (i\geq 0)  \text{ and } \ad\p_{2,-1} (\p_{i,j}) = b \, \p_{i+2,j-1} \, (j\geq 0)
\] 
with non-zero constants $a,b \in \kk$. This implies that the Lie algebra generated by $\p_{-1,2}$ and $\p_{2,-1}$ 
contains all $\p_{i,j}$ with $(i,j) \in \Lambda_0$, hence  $\langle \Lie U, \Lie V\rangle_{\text{Lie}} \supset (\Lie \A)^S$, 
see formula \eqref{Lie-AS.eq}. Since $\Lie U, \Lie V \subset \Lie \F$ the claim follows from the inclusions in \eqref{G-equ}.
\end{proof}
It is easy to see that the vector field $\p_{i,j}$ is locally nilpotent if and only if $i = -1$ or $j=-1$. In these cases, 
there are unique algebraic subgroups $U_{-1,k} \simeq \Ga$ and $U_{k,-1} \simeq \Ga$ with $\Lie U_{-1,k} = \kk\p_{-1,k}$ and $\Lie U_{k,-1} = \kk\p_{k,-1}$. 
Note that $U_{-1,2} = U =\overline{\langle\bu\rangle}$ and $U_{2,-1} = V =\overline{\langle\bv\rangle}$. 
Moreover, the subgroups $U_{-1,k}$ and $U_{-1,l}$ commute as well as the vector fields $\p_{-1,k}$ and $\p_{-1,l}$.
\ps
\subsection{De Jonqui\`eres subgroups and amalgamated products}\label{Jonquiere.subsec} 
This and the subsequent sections are largely inspired by \cite[Section~4]{ArZa2013Acyclic-curves-and}.
Consider the  de Jonqui\`eres group 
\[
\Jonq:= \{(ax+h(y), cy + d) \mid a,c\in\kk^*, d\in \kk, h \in \kk[y]\}
\]
and its subgroups
\begin{gather*}
\Jonq_0:= \{(ax+yf(y),by) \mid a,b \in \kk^*, f\in\kk[y]\} = \Jonq \cap \Aut_0(\Atwo), \\
\J_0:=\{(x+yf(y),y) \mid f\in \kk[y]\} \subset \Jonq_0. 
\end{gather*}
$\J_0$ is a  commutative closed unipotent ind-subgroup, isomorphic to $\kk[y]^+$. 
It admits a filtration by closed unipotent subgroups $\prod_{i=1}^d U_{-1,i}\simeq \Ga^{d}$, $d\geq 1$. Note that $\Jonq_0 = \J_0 \rtimes T$.
\begin{lem}\label{free-product.lem}
The group $\A^S$ is the free product $(\J_0)^S * (\J_0^-)^S$ where   $\J_0^-:= \tau\circ \J_0 \circ
\tau$ and $\tau:=[\begin{smallmatrix} 0 & 1 \\ 1 & 0 \end{smallmatrix}]$.
\end{lem}
\begin{proof}
(1) 
The amalgamated product structure $\Aut(\Atwo) = \Jonq*_\B \Aff_2$ where $\B:=\Jonq \cap \Aff_2$ 
implies that every element $\bg \in \Aut(\Atwo)\setminus\B$ can be expressed in one of the following 4 forms, called {\it the type} of $\bg$,
\begin{gather*}
\bg = j_1 a_1 j_2 \cdots  j_n, \qquad \bg = j_1a_1j_2 \cdots j_n a_{n},\\
\bg = a_1 j_1 a_2  \cdots j_n, \qquad \bg = a_1 j_1 a_2 \cdots j_n a_{n+1}
\end{gather*}
where $j_k\in\Jonq\setminus\B$ and $a_k \in \Aff_2\setminus\B$. 
The length of the expression, its type and the degrees of the $j_k$ are uniquely determined by $\bg$.
\ps
(2) 
Let now $\bg \in \Aut_0(\Atwo)$. Using $\Jonq=\Jonq_0\rtimes \Trans_2$ and  
$\Aff_2 = \GL_2 \rtimes \Trans_2$ where $\Trans_2 \simeq \Ga^2$ is the group of translations, one can assume that $j_k \in \Jonq_0\setminus B$ 
and $a_k \in \GL_2\setminus B$ where $B \subset \GL_2$ is the Borel subgroup of upper triangular matrices. 
Since $\GL_2 \setminus B = U \tau B = B \tau U$ where $U \subset B$ is the unipotent radical, we see that any $\bg \in \Aut_0(\Atwo)$ has a presentation in one of the 4 forms
\begin{gather*}
\bg = j_1 \tau j_2 \cdots \tau j_n, \qquad \bg = j_1 \tau j_2 \cdots j_n \tau,\\
\bg = \tau j_1 \tau  \cdots \tau j_n, \qquad \bg = \tau j_1 \tau \cdots \tau j_n \tau
\end{gather*}
where $j_k \in \Jonq_0 \setminus B$. Again, the length, the type and the degrees of the $j_k$ are uniquely determined by $\bg$. 
\ps
(3)
Now we remark that 
$\Jonq_0=\mathfrak{J}_0\rtimes T$ and 
$T\tau = \tau T$. It follows that 
we can reach one of the following forms
\begin{gather*}
\bg = j_1 \tau j_2 \cdots \tau j_n t, \qquad \bg = j_1 \tau j_2 \cdots j_n \tau t,\\
\bg = \tau j_1 \tau  \cdots \tau j_n t, \qquad \bg = \tau j_1 \tau \cdots \tau j_n \tau t
\end{gather*}
where $j_k \in \J_0$ and $t \in T$. We claim that this form is uniquely determined by $\bg$. 
Since the type and the length are given by $\bg$, it suffices to compare the similar decompositions
$$
\bg = j_1 \tau j_2 \cdots j_n t = j_1' \tau j_2' \cdots j_n' t'  \text{ and  }
\bg = j_1 \tau j_2 \cdots j_n \tau t = j_1' \tau j_2' \cdots j_n' \tau t'.
$$ 
It follows that
$$
\tau j_2 \cdots j_n t = j_1^{-1}j_1' \tau j_2' \cdots j_n' t'  \text{  \ resp. \ }
\tau j_2 \cdots j_n \tau t = j_1^{-1} j_1' \tau j_2' \cdots j_n' \tau t',
$$
hence $j_1' = j_1$. Now the claim follows by induction.
\ps
(4) 
Using the equalities $st = ts$ and $s\tau s = \tau$ for $s \in S$ and $t \in T$ we find for an $S$-invariant $\bg\in\Aut_0(\Atwo)$ 
of the first form $\bg = j_1 \tau j_2 \cdots j_n t$ in (3) above that
$$
\bg = s \bg s^{-1} = 
\begin{cases}
(s j_1 s^{-1}) \tau (s^{-1} j_2 s) \tau (s j_3 s^{-1})\cdots (s j_n s^{-1}) t & \text{if $n$ is odd,}\\
(s j_1 s^{-1}) \tau (s^{-1} j_2 s) \tau (s j_3 s^{-1})\cdots (s^{-1} j_n s^{-1}) t &  \text{if $n$ is even.}
\end{cases}
$$
In the first case, we get that each $j_k$ is $S$-invariant. In the second case, we obtain  $s j_n s = j_n$, a contradiction. 
Looking at the other forms of $\bg$ in (3) we see that for an $S$-invariant $\bg$  the number of $\tau$'s must be even and each $j_k$ is $S$-invariant, i.e.  $j_k \in (\J_0)^S \subset \A^S$.
\ps
(5)
Finally, assume that $\bg \in \A^S$. If $\bg$ is given in one of the forms of (3), then $t = \id$, because 
$d\bg_0 = \id$ and  $(d j_k)_0 = \id$ for $j_k \in (\J_0)^S$. Since the number of $\tau$'s is even, 
we see that such an automorphism $\bg$ is a product of the form $ \cdots j_k j_k' j_{k+1}j_{k+1}' \cdots$ where $j_k \in (\J_0)^S$ 
and $j_k' \in (\J_0^-)^S = (\tau \J_0 \tau)^S = \tau(\J_0)^S \tau$.
\end{proof}
We have seen in the proof above that every $\bg \in \Aut_0(\Atwo)$ has a unique presentation in one of the 4 forms
\begin{gather*}
\bg = j_1 \tau j_2 \cdots \tau j_n, \qquad \bg = j_1 \tau j_2 \cdots j_n \tau,\\
\bg = \tau j_1 \tau  \cdots \tau j_n, \qquad \bg = \tau j_1 \tau \cdots \tau j_n \tau
\end{gather*}
where $j_k \in \Jonq_0 \setminus B$. As usual, we define the {\it degree\/} of an element $(f,h)\in\Aut(\Atwo)$ as $\deg (f,h):=\max\{\deg f, \deg h\}$. 
The following lemma is known, see e.g. \cite[Lemma~4.1]{Ka1979Automorphism-group}.
\begin{lem}\label{degree.lem}
If $\bg$ is as above, then $\deg \bg = \deg j_1 \cdot \deg j_2 \cdots \deg j_n$.
\end{lem}%
\begin{proof}
It suffices to consider the cases $\bg = \tau j_1 \tau  \cdots \tau j_n = (f_n,h_n)$. We now prove by induction that $\deg f_k > \deg h_k$ 
and that $\deg f_{k+1} = \deg f_k \cdot \deg j_{k+1}$. Clearly, $\deg f_1 > \deg h_1=1$, because $j_1 \notin B$. If $j_{k+1} = (x,y+p(x)$, then 
$$
(f_{k+1},h_{k+1}) = \tau\circ j_{k+1} \circ (f_k,h_k) = (h_k + p(f_k), f_k).
$$
Since $\deg j_{k+1} = \deg p >1$ and $\deg h_k < \deg f_k$ we get $\deg f_{k+1} = \deg p \cdot \deg f_k > \deg f_k$, hence the claim.
\end{proof}

\ps
\subsection{The fixed points under \tp{$S$}{S}}\label{fixed-points.subsec}
Recall that $\F \subset \Aut(\Atwo)$ denotes the closure of the free group $F$ generated by $\bu$ and $\bv$.
\begin{prop}\label{main.prop}
$(\J_0)^S \subset \F$. In particular, $\F = \A^S = (\J_0)^S * (\J_0^-)^S$.
\end{prop}
\begin{proof}
We have seen in Section~\ref{Jonquiere.subsec} that $\J_0 \simeq \kk[y]^+$ admits a filtration by the closed unipotent subgroups $\prod_{i=2}^d U_{-1,i}$, $d\geq 2$. 
It follows that $\J_0^S = \{(x + y^2 f(y^3), y) \mid f \in \kk[y]\} \simto \kk[y]^+$, and that $\J_0^S$ admits a filtration by the closed unipotent subgroups $\prod_{j=1}^k U_{-1,3j-1}$. 
Note that this product contains any automorphism of $\Atwo$ of the form $(x + \sum_{j=1}^{k} c_{j}y^{3j-1},y)$.
We now show by induction that $U_{-1,3j-1} \subset \F$.

Assume that this holds for $j\leq k$. Set
\begin{multline*}
\phi:=\bu\circ \bv =
(x+y^{2},y) \circ (x,y+x^{2}) = 
\\
= (x+y^{2} + 2x^2 y + x^4, y + x^{2})\in F=\langle \bu,\bv\rangle
\end{multline*}
and consider the product 
$$
\phi\circ (x -\sum_{j=1}^k c_j y^{3j-1}, y) =: (f_k,h_k) \in F \cdot\prod_{j=1}^kU_{-1,3j-1}.
$$
Then
$$
\textstyle f_k(0,y) = - \sum_{1}^k c_j y^{3j-1} + y^2 + 2  y(\sum_{1}^k c_j y^{3j-1})^2 +  
(\sum_{1}^k c_j y^{3j-1})^4.
$$
In particular, the exponents of $y$ in $f_k(0,y)$ are all $\equiv 2\mod 3$.
Setting $c_1=1$ the second power of $y$ disappears and we get
$$
\textstyle f_k(0,y) = - \sum_{2}^k c_j y^{3j-1} + 2  y(\sum_{1}^k c_j y^{3j-1})^2 +  
(\sum_{1}^k c_j y^{3j-1})^4.
$$
Hence, for $j\leq k$, the coefficient $d_j$ of $y^{3j-1}$ has the form
$$
d_j = 
-c_j + 2 \sum_{j_1+j_2=j} c_{j_1} c_{j_2} + \sum_{j_1+j_2+j_3+j_4=j+1} c_{j_1}c_{j_2}c_{j_3}c_{j_4}.
$$
It follows that for any $j\leq k$ the system of equations $d_1=d_2=\cdots=d_j=0$ admits a unique solution
 in positive integers $c_i$  which does not depend on $k$ ($c_2 = 2, c_3=9, c_4=52,c_5=340, c_6 = 2394, \ldots$).
Moreover, we see that the coefficient $d_{k+1}$ 
of $y^{3k+2}$ is positive. But then
\[
(t^{-3k-2}x, t^{-1}y)\circ (f_k,h_k) \circ (t^{3k+2}x,t y) \overset{t \to 0}{\longrightarrow}
(x + d_{k+1}y^{3k+2}, y) \in U_{-1,3k+2}.
\]
Since $\F$ is closed and stable under the action of the torus $T$ we get $U_{-1,3k+2} \subset \F$, and the claim follows. 
\end{proof}
Summing up we have the following result. 
\begin{thm}\label{example.thm}
Let $F \subset \Aut(\Atwo)$ be the subgroup generated by $\bu:=(x+y^2,y)$ and $\bv:=(x,y+x^2)$, and denote by $\F:=\overline{F}$ 
its closure in $\Aut(\Atwo)$. Furthermore, define $\A:=\{\bg \in \Aut(\Atwo) \mid \bg(0) = 0, d\bg_0 = \id\}$.
\be
\item 
$F$ is a free group in the two generators $\bu,\bv$, containing elements which are not locally finite.
\item
$\F = \A^S = \Cent_\A(S)$ where $S:=\{(\zeta x, \zeta^2 y)\mid \zeta^3=1\} \subset T$ is the centralizer of $F$ in $\Aut(\Atwo)$.
\item
$\F$ is a free product $\J * \J^-$ where $\J := \{(x+y^2 f(y^3), y) \mid f\in\kk[y]\} = (\J_0)^S$ and 
$\J^- = \{(x, y + x^2 f(x^3)) \mid f\in\kk[x]\}=\tau\J\tau$. In particular, $\J$ is a commutative nested unipotent closed ind-subgroup of $\Aut(\Atwo)$ isomorphic to $\kk[y]^+$.
\item
Every algebraic subgroup of $\F$ is commutative and unipotent and is conjugate to a subgroup of $\J$ 
or of $\J^-$.
\ee
\end{thm}
\begin{proof}
(1) 
The first part is clear. For the second we
notice that $\deg(\bu\bv)^k$ tends to infinity for $k\to\infty$. Hence, $\bu\bv \in \F$ is not locally finite.
\ps
(2) \& (3)
This is Proposition~\ref{main.prop} and Lemma~\ref{example.lem}.
\ps
(4)
For an algebraic subset $X \subset \F$ the degrees of the elements from $X$ are bounded above. It follows from Lemma~\ref{degree.lem} that the elements of $X$ have bounded length in $\F$. Hence  
a famous result of \name{Serre}'s (see \cite[Theorem~8]{Se2003Trees}) implies that
an algebraic subgroup of the free product $\F = \J * \J^-$ is
conjugate to a subgroup of one of the factors.
\end{proof}
\begin{rems}\label{example.rems}
\be
\item
According to Lemma 3.2.1 the subgroup $\F \subset \Aut(X)$ is the closure of the subgroup generated by $U$ and $V$, 
and $\Lie\F$ is the Lie algebra generated by $\Lie U$ and $\Lie V$. Thus our example is a positive instance of our Question~\ref{Q1}.

\item
The free product $U*V$ is strictly contained in $\F$, because, by Lemma~\ref{degree.lem}, the degree of an element from $U*V$ is a power of 2, and so $U*V$ cannot contain all the groups $U_{-1,3j-1}$. Even more, by the same lemma, every algebraic subgroup of $U*V$ is of bounded length, hence conjugate to either $U$ or $V$ by \name{Serre}'s result \cite[Theorem~8]{Se2003Trees}).

\item\label{not-universal.item}
The theorem above shows that there is an ind-structure on the free product of the two ind-groups $\J$ 
and $\J^-$. However, this structure {\it is not universal\/}, i.e. it is not true that 
given two homomorphisms of ind-groups $\J \to \G$ and $\J^- \to \G$, then the induced homomorphism $\J*\J^- \to \G$ is a homomorphism of ind-groups.

In order to see this consider the two projections $\J \to U \subset \Aut(\Atwo)$ and $\J^- \to V \subset \Aut(\Atwo)$. We claim that the induced homomorphism $\phi\colon \J*\J^- \to \Aut(\Atwo)$ is not an ind-homomorphism. In fact, the image of $\phi$ is $U*V \subset \J*\J^-$ and $\phi$ induces the identity on $U*V$. Since $U*V$ is dense in $\J*\J^-$ it follows that $\phi$ is the identity on $\J*\J^-$, a contradiction.

\ee
\end{rems}

\ps
%
\section{A canonical Lie algebra for an ind-group}\label{canonical-LA.sec}
For an ind-variety $\V$ the tangent space $T_x\V$ in a point $x \in \V$ can be defined as the limit
$\varinjlim T_x X$ where $X$ runs through all algebraic subsets of $\V$ containing $x$. There is   a canonical subspace of $T_x\V$ generated by those $T_xX$ where $X$ is smooth in $x$:
$$
T^{(s)}_x \V :=\Span(T_xX \mid X \subset \V \text{ algebraic and }x \in X_\reg).
$$
If $\V$ is {\it strongly smooth in $x$}, i.e. there is an admissible filtration $\V = \bigcup_k \V_k$ such that each $\V_k$ is smooth in $x$, then $T^{(s)}_x\V = T_x \V$. 
In general, we could not prove any useful results for this subspace, but for ind-groups it turns out that it has very remarkable properties.
\ps
\subsection{Definition and first properties}
\begin{defn} For an ind-group $\G$ we define
$$
L_\G : = \Span(T_eX \mid X \subset \G \text{ algebraic and }e \in X_\reg) \subset \Lie\G.
$$
\end{defn}
The next proposition shows that $L_\G \subset \Lie \G$ is a canonical Lie subalgebra.
\begin{prop}\label{LG.prop}\strut
\be
\item[{\rm (a)}]
$L_{\G} \subset \Lie \G$ is stable under the adjoint action of $\G$, i.e. $L_{\G}$ is
an ideal in $\Lie\G$.
\item[{\rm (b)}]
For every algebraic subgroup $G\subset \G$ we have $\Lie G \subset L_{\G}$.
\item[{\rm (c)}]
$L_{\G^\circ} = L_\G$.
\item[{\rm (d)}]
$L_\G$ is finite dimensional if and only if $\G^\circ$ is an algebraic group, and in this case we have $L_\G = \Lie \G$.
\ee
\end{prop}
\begin{proof}
(a) If $X \subset \G$ is an algebraic subset with $e \in X_\reg$ and if $g \in \G$, then $X':=g X g^{-1}$ is 
an algebraic subset and $e \in X'_\reg$. Moreover, $\Ad g (T_e X) = T_e X' \subset L_\G$, and the claim follows. 
\ps
(b) and (c) are clear.
\ps
(d) If $\G^\circ$ is an algebraic group, then the claims follow from (b) and (c). If $L_\G$ is finite dimensional, then so is $L_{\G^\circ}$ by (c). 
As a consequence, the dimension of an irreducible subvariety $X \subset \G^\circ$ is bounded by $\dim L_{\G^\circ}=\dim L_\G$, 
and the claim follows from Proposition~\ref{Ramanujam.prop}.
\end{proof}

The following results show that one has $L_\G = \Lie\G$ in many cases, but we have no example where $L_\G \neq \Lie \G$.
\begin{cor}\label{LG.cor}
\be
\item 
Let $\G$ be algebraically generated by a family $\{G_i\}_{i\in I}$ 
of connected algebraic subgroups. Then $L(\G)\subset L_{\G}$ is an ideal. 

\item 
If $\G$ is strongly smooth in $e$, then $L_{\G}=\Lie(G)$. 
This holds, in particular, for a nested ind-group $\G$. 
\ee
\end{cor}

Recall that
\[
\SAut(\AA^n):=\{\alpha\in\Aut(\AA^n)\,|\,{\rm jac}(\alpha)=1\}
\]
is a normal subgroup of $\Aut(\AA^n)$ with Lie algebra $\Lie \SAut(\AA^n) = \VF^0(\AA^n)$, 
the algebra of polynomial vector fields on $\AA^n$ with zero divergence (see Section~\ref{Lie-algebra.subsec}).
\begin{cor}\label{AutAn.cor}
We have $L_{\SAut(\An)} = \Lie\SAut(\An)$ and $L_{\Aut(\An)} = \Lie\Aut(\An)$.
\end{cor}
Note that $\Aut(\Atwo)$ is not strongly smooth in $e$ (see \cite[Corollary~14.1.2]{FuKr2018On-the-geometry-of}) and so we cannot apply Corollary~\ref{LG.cor}(2).
\begin{proof} By  \cite[Lemma~3]{Sh1981On-some-infinite-d}, the Lie algebra 
\[
\Lie\SAut(\An)=[\Lie\Aut(\An),\Lie\Aut(\An)]
\] 
is simple, and by  Proposition~\ref{LG.prop}(a) $L_{\SAut(\An)}$ is a (non-zero) ideal in $\Lie\SAut(\An)$. 
This proves the first claim.
\ps
The second claim is a consequence of Proposition~\ref{LG.prop}(b) due to the fact  that
\[
\Lie \Aut(\An) = \Lie\SAut(\An) \oplus \Lie\kk^*,
\] 
where $\kk^*\subset\Aut(\An)$ 
is realized as the subgroup of  the scalar multiplications $\lambda\cdot\id$, 
$\lambda\in\kk^*$, cf.~\cite[Prop.~15.7.2]{FuKr2018On-the-geometry-of}.
\end{proof}
\ps
\subsection{Functorial properties}
\begin{prop}\label{functorial.prop}
Let $\phi\colon \H \to \G$ be a homomorphism of ind-groups. 
\be
\item
Then $d\phi_e(L_\H) \subset L_\G$. In particular, $L_\G$ is invariant under every ind-group automorphism $\G \simto \G$.
\item
Assume that $\kk$ is uncountable. If $\phi$ is surjective, then the differential $d\phi_e\colon L_\H \to L_\G$ is surjective.
\ee
\end{prop}
\begin{proof}
(a) Let $Y \subset \H$ be an irreducible algebraic subset such that $e \in Y_\reg$. 
There is an open dense subset $U \subset Y_\reg$ such that $\phi(U) \subset \overline{\phi(Y)}_\reg$ 
and that $\phi|_U\colon U \to \phi(U)$ is smooth in every point $u \in U$. In particular, for all $u \in U$, the image  $\phi(Y\!u^{-1})$ is smooth in $e$, hence the image of $T_e \,Y\!u^{-1}$ is contained in $L_\G$ for all $u \in U$. Now the claim follows from Lemma~\ref{dense-subset.lem} below, because $U$ is dense in $Y_\reg$.
\ps
(b) Let $X \subset \G$ be an irreducible algebraic subset such that $e \in X_\reg$. We have to show that $T_e X \subset d\phi_e(L_\H)$. 
Since $\kk$ is uncountable there exists an irreducible closed algebraic subset $Y \subset \H$ such that $\phi(Y) = X$ (\cite[Prop.~1.3.2]{FuKr2018On-the-geometry-of}). As in (a) we choose a dense open $U \subset Y$ such that $\phi(U)$ is an open subset of $X_\reg$ and that $\phi|_U \colon U \to \phi(U)$ is smooth. 
Then $d\phi_e (T_e \,Y\!u^{-1}) = T_e\,X\phi(u)^{-1}$ for all $u \in U$. It follows that $T_e\,Xv^{-1} \subset d\phi_e (L_\H)$ for all $v \in \phi(U)$, and hence for all $v \in X_\reg$, again by Lemma~\ref{dense-subset.lem} below.
\end{proof}
\begin{rem}
The assumption that $\kk$ is uncountable is necessary, as seen from the following example. Take the bijective homomorphism $\kk^+ \to \Ga$ where $\kk^+$ is endowed with the discrete topology.
\end{rem}
\begin{lem}\label{dense-subset.lem}
Let $Y \subset \H$ be an irreducible algebraic subset. Then 
$\Span(T_e \,Y\!g^{-1} \mid g\in Y) \subset \Lie\H$ is finite dimensional, and 
$$
\Span(T_e \,Y\!g^{-1} \mid g\in Y_\reg) = \Span(T_e \,Y\!g^{-1} \mid g\in U)
$$
for any dense subset $U \subset Y_\reg$.
\end{lem}
\begin{proof}
Set $X:=\overline{YY^{-1}} \subset \G$. Then
$T_e \,Y\!g^{-1} \subset T_e X$ for all $g \in Y$ which implies the first claim.
\ps
For the second claim let $\Grass_n T_eX$ denote the Grassmannian of $n$-dimensional subspaces of $T_eX$ where $n := \dim Y$, and consider the morphism $\mu\colon Y_\reg \to \Grass_n T_eX$ given by $g \mapsto T_e \, Y\!g^{-1}$. Since $\Grass_n V \subset \Grass_n T_e X$ is closed for any subspace $V \subset T_eX$ it follows that a dense subset of $Y_\reg$ has the same span as $Y_\reg$.
\end{proof}
\begin{que}\label{question.surj}
Assume we have a surjective homomorphism $\G \to G$ where $\G$ is a path-connected ind-group,
$G$ an algebraic group and $\kk$ uncountable.
Does $\G$ contain an algebraic
subgroup which is sent surjectively  onto $G$?
\end{que}
\ps
\subsection{Subgroups of finite codimension}
Let $\G=\bigcup_d \G_d$ be an ind-group.
\begin{defn} 
We say that a closed subgroup $\H \subset \G$ is  {\it of finite codimension\/} if
$\G_d \cdot \H = \G$ for some $d\in\NN$. 
\end{defn}
For instance, if $\G$ is connected and acts on an algebraic variety $X$, then the stabilizer
$\G_x$ of any point $x\in X$ is of finite codimension. Indeed, by \cite[Prop.~7.1.2]{FuKr2018On-the-geometry-of} the orbit $\G x$ is an irreducible subvariety of $X$  and $\G x = \G_{d}x$
for some $d$. The latter implies $\G = \G_{d}\cdot \G_{x}$.

We do not know if every subgroup $\H$ of finite codimension of a connected ind-group $\G$
is the stabilizer of a point of a $\G$-variety, i.e. if $\G/\H$ has a canonical structure
of an algebraic variety. 

\begin{exa}
The connected component 
$\G^\circ$ has finite codimension in $\G$  if and only if it has finite index in $\G$. Indeed, a closed algebraic subset $Z \subset \G$ 
can only meet finitely many of the components of $\G$. Similarly, if $\H$ has finite codimension in $\G$, then the image of $\H$ in $\G/\G^\circ$ has finite index.
\end{exa}
One could expect that for a subgroup $\H \subset \G$ of finite codimension the Lie algebra $\Lie\H \subset \Lie \G$ 
has finite codimension as well. What we can show is the following result.

\begin{prop}\label{finite-codimension.prop} Assume that the base field $\kk$ is uncountable.
Let $\G$ be an  ind-group and $\H \subset \G$ a closed subgroup of finite codimension. 
If $\H$ is subnormal in $\G$, i.e. there is a finite series $\H=\H_0 \subset \H_1\subset \cdots\subset \H_\ell = \G$ 
such that $\H_i$ is normal in $\H_{i+1}$, then $L_{\H}$ has finite codimension in $L_{\G}$. 
\end{prop}

The proof needs some preparation. For an algebraic subset $X \subset \G$ we define
$$
M_X:=\Span(T_e \,X\!g^{-1} \mid g \in X) \subset \Lie\G.
$$
We have seen in Lemma~\ref{dense-subset.lem} that $M_X$ is finite dimensional. Moreover, if $X$ is smooth, 
then $M_U = M_X$ for an open dense subset $U$ of $X$, and $M_X \subset L_\G$.

\begin{lem}\label{finite-codimension.lem} 
Assume that $\kk$ is uncountable.
Let $\G$ be an ind-group and $\H \subset \G$ a closed normal subgroup. If there exists a
closed irreducible algebraic subset $Z \subset \G$ such that $Z\cdot \H = \G$,
then $L_{\G} \subset M_{Z} + L_{\H}$. In particular, $L_{\H}$ has finite codimension in $L_{\G}$.
\end{lem}
\begin{proof}

(1)
Let $X \subset \G$ be a closed irreducible algebraic subset. Since $\kk$ is uncountable we can find an irreducible algebraic subset 
$Y \subset \H$ such that $Z\cdot Y \supset X$. Thus $X$ is contained in the image of the multiplication morphism 
$\mu\colon Z \times Y \to \G$, $\mu(z,y) := z\cdot y$. For every $g:= z \cdot y$
where $(z,y)\in Z \times Y$ we have the following commutative diagram
\[
\begin{CD}
Z \times Y @>\mu>> \G\\
@V{\rho_{z^{-1}}\times \Int z \, \circ \, \rho_{y^{-1}}}V{\simeq}V  @V{\simeq}V{\rho_{g^{-1}}}V\\
Zz^{-1}\times z(Y y^{-1})z^{-1}  @>\mu>> \G
\end{CD}
\]
which induces the commutative diagram of tangent maps
\[
\begin{CD}
T_{z}Z \oplus T_{y}Y @>{d\mu_{(z,y)}}>> T_{g}\G\\
@V{d\rho_{z^{-1}}\times \Ad z \, \circ \, d\rho_{y^{-1}}}V{\simeq}V  @V{\simeq}V{d\rho_{g^{-1}}}V\\
T_{e}(Zz^{-1})\oplus T_{e}(z(Y y^{-1})z^{-1})  @>d\mu_{(e,e)}>> T_{e}\G
\end{CD}
\]
(2) There exists an irreducible open subset $U \subset \mu^{-1}(X)_\reg$ such that $\mu(U) \subset X_\reg$ 
is open and dense, and that $\mu|_U \colon U \to X$ is smooth. Replacing $Y$ by $\overline{\pr_Y(U)}$ if necessary 
we can assume that $\pr_Y(U) \subset Y_\reg$. For $g = z\cdot y$ where $(z,y) \in U$ 
it follows from the diagram above that 
$$
T_e \,X\!g^{-1} \subset T_e \, Zz^{-1} + \Ad z(T_e \, Y\!y^{-1}) \subset M_Z + \Ad z(L_\H),
$$
hence $M_{X_\reg} \subset M_Z + L_\H$ by Lemma~\ref{dense-subset.lem}. Since the $M_{X_\reg} \subset L_\G$ generate $L_\G$ the claim follows.
\end{proof}

\begin{proof}[Proof of Proposition~\ref{finite-codimension.prop}]
We can clearly assume that $\H$ is normal in $\G$. 

(1) We first claim that $\H \cap \G^\circ$ has finite codimension in $\G^\circ$. The ind-group $\G$ is a countable disjoint union 
$\G = \bigcup_{j\in J} g_j \G^\circ=\bigcup_{j\in J} \G^\circ g_j$ where the components are open, closed, and connected, cf. \cite[Section~2.2]{FuKr2018On-the-geometry-of}. 
Similarly we get a countable disjoint union of open and closed subsets of $\H$ in the form $\H = \bigcup_{i \in I} h_i (\H\cap\G^\circ)$.

Assume that $Z \cdot \H = \G$ for some algebraic subset $Z \subset \G$. Then $Z = \bigcup_{j\in J_0} Z_j g_j$ for a finite subset $J_0 \subset J$ and algebraic subsets $Z_j \subset \G^\circ$. It follows that 
$$
\G=Z\cdot \H = \bigcup_{j\in J_0} Z_j g_j \cdot \H \text{ \ and \ } Z_jg_j\cdot \H = \bigcup_{i \in I} Z_j g_j h_i \cdot (\H\cap \G^\circ).
$$
Each $Z_j g_j h_i \cdot (\H\cap\G^\circ)$ belongs to a connected component of $\G$, and $Z_j g_j h_i \cdot (\H \cap \G^\circ) \subset \G^\circ$ if and only if $g_j h_i \in \G^\circ$. Thus
$$
\G^\circ =\bigcup_{j\in J_0, \, g_jh_i\in \G^\circ} Z_j g_j h_i \cdot (\H \cap \G^\circ) = 
(\bigcup_{j\in J_0, \, g_jh_i\in \G^\circ} Z_j g_j h_i) \cdot (\H \cap \G^\circ).
$$
By construction, $\bigcup_{j\in J_0, \, g_jh_i\in \G^\circ} Z_j g_j h_i$ is a finite union of algebraic subsets and thus contained in $\G^\circ_k$ for some $k$ where $\G^\circ = \bigcup_k \G^\circ_k$ is an admissible filtration.
\ps
(2) 
Since $\G^\circ$ is connected we can assume that the $\G^\circ_k$ are irreducible (\cite[Prop.~1.6.3 and Prop.~2.2.1(2)]{FuKr2018On-the-geometry-of}) and so $Z' \cdot (\H \cap \G^\circ) = \G^\circ$ with an irreducible algebraic subset $Z' \subset \G^\circ$. Now Lemma~\ref{finite-codimension.lem} and Proposition~\ref{LG.prop}(c) imply that $L_{\H} = L_{\H \cap \G^\circ}$ has finite codimension in $L_\G = L_{\G^\circ}$.
\end{proof}

\textbf{Acknowledgement.}
We would like to thank the referee for his very careful reading of the paper and his comments which helped to improve the paper considerably. 
Our thanks also go to \name{Alexander Perepechko} for his helpful remarks.  
A part of the paper was written during a stay of the second author at the Max Planck Institute of Mathematics in Bonn. 
He thanks this institution for its support and excellent working conditions.

\par\bigskip
\renewcommand{\MR}[1]{}

\end{document}